\documentclass[12pt]{article}
\usepackage{mathtools}
\usepackage{bbm}
\mathtoolsset{showonlyrefs}

\newcommand{\old}[1]{}

\renewcommand{\emph}[1]{\textit{#1}}
\usepackage{enumerate,amsmath,latexsym,amssymb,amsthm}
\usepackage{color}\usepackage{graphicx}
\usepackage{enumitem}

\newsavebox{\mycases}

 \allowdisplaybreaks

\newcounter{rot}

\newcommand{\ignore}[1]{}

\def\cA{{\cal A}}

\newcommand{\set}[1]{\left\{#1\right\}}

\newcommand{\proofstart}{\begin{proof}}
\newcommand{\proofend}{\end{proof}}

\def\ii_(#1,#2){i_{#1}^{#2}}

\def\a{\alpha}
\def\b{\beta}

\def\D{\Delta}
\def\e{\varepsilon}
\def\f{\phi}

\def\g{\gamma}

\def\k{\kappa}

\def\z{\zeta}

\def\th{\theta}

\def\l{\lambda}
\def\m{\mu}
\def\n{\nu}
\def\p{\pi}

\def\r{\rho}

\def\s{\sigma}
\def\t{\tau}
\def\om{\omega}

\newcommand{\rdown}[1]{\mbox{$\left\lfloor #1 \right\rfloor$}}

\def\cD{\mathcal{D}}
\def\cE{\mathcal{E}}

\def\cH{{\mathcal H}}

\def\Re{\mathbb{R}}
\newcommand{\brac}[1]{\left( #1 \right)}

\newcommand{\expect}{\operatorname{\bf E}}
\def\E{\expect}
\renewcommand{\Pr}{\operatorname{\bf Pr}}
\newcommand\bfrac[2]{\left(\frac{#1}{#2}\right)}

\newtheorem{theorem}{Theorem}[section]
\newtheorem{conjecture}[theorem]{Conjecture}
\newtheorem{lemma}[theorem]{Lemma}
\newtheorem{corollary}[theorem]{Corollary}

{
\theoremstyle{definition}
\newtheorem{remark}[theorem]{Remark}
}
\newtheorem{claim}[theorem]{Claim}

\newcounter{thmtemp}

\newcommand{\nospace}[1]{}

\def\path{\operatorname{PATH}}

\newcommand{\erl}{\mathrm{ERL}}
\newcommand{\card}[1]{\left| #1 \right|}

\newcommand{\beq}[1]{\begin{equation}\label{#1}}
\def\eeq{\end{equation}}

\renewcommand{\Re}{\mathbb{R}}

\def\PAG{Preferential Attachment Graph}
\def\leb{\leq_b}
\def\cL{{\cal L}}

\def\La{\Lambda}
\def\vt{{v_{t}}}
\def\vtt{{v_{t+1}}}
\def\Ht{{H_{t}}}
\def\ff{}
\def\leb{\leq_b}
\newcommand{\lal}[2]{\l^{(#1)}_{#2}}
\newcommand{\vv}[2]{v^{(#1)}_{#2}}
\newcommand{\ww}[2]{x^{(#1)}_{#2}}
\newcommand{\UU}[2]{U^{(#1)}_{#2}}
\newcommand{\WW}[2]{X^{(#1)}_{#2}}
\newcommand{\vun}[1]{\mathbf{1}_{#1}}

\begin{document}
\title{Looking for vertex number one}

\author{Alan Frieze\thanks{Research supported in part by NSF grant
    CCF1013110}, Wesley Pegden\thanks{Research supported in part by NSF grant
    DMS}\\
Department of Mathematical Sciences,\\
Carnegie Mellon University,\\
Pittsburgh PA 15213,\\
U.S.A.}

\maketitle
\begin{abstract}
Given an instance of the preferential attachment graph $G_n=([n],E_n)$, we would like to find vertex 1, using only `local' information about the graph; that is, by exploring the neighborhoods of small sets of vertices.  Borgs et al. gave an an algorithm which runs in time $O(\log^4 n)$, which is local in the sense that at each step, it needs only to search the neighborhood of a set of vertices of size $O(\log^4 n)$.  
We give an algorithm to find vertex 1, which w.h.p. runs in time $O(\om\log n)$ and which is local in the strongest sense of operating only on neighborhoods of single vertices. Here $\om=\om(n)$ is any function that goes to infinity with $n$.
\end{abstract}
\section{Introduction}
The \PAG\ $G_n$ was first discussed by Barab\'asi and Albert \cite{BA} and then rigorously analysed by Bollob\'as, Riordan, Spencer and Tusn\'ady \cite{BRST}. It is perhaps the simplest model of a natural process that produces a graph with a power law degree sequence.

The \PAG\ can be viewed as a sequence of random graphs $G_1,G_2,\ldots,G_n$
where $G_{t+1}$ is obtained from $G_t$ as follows: Given $G_t$, we add
vertex $t+1$ and $m$ random edges $\set{e_i=(t+1,u_i):\;1\leq i\leq
  m}$ incident with vertex $t+1$. Here the constant $m$ is a parameter of the
model. The vertices $u_i$ are not chosen uniformly from $V_t$, instead
they are chosen with probabilities proportional to their degrees. This tends to generate
some very high degree vertices, compared with what one would expect in Erd\H{o}s-R\'enyi
models with the same edge-density. We refer to $u_1,u_2,\ldots,u_m$ as the {\em left} choices of vertex $t+1$. We also say that $t+1$ is a {\em right} neighbor of
$u_i$ for $i=1,2,\ldots,m$. 

We consider the problem of searching through the preferential attachment graph looking for vertex number 1, using only local information. This was addressed by Borgs, Brautbar, Chayes, Khanna and Lucier \cite{BBCKL} in the context of the Preferential Attachment Graph $G_n=(V_n,E_n)$. Here $V_n=[n]=\set{1,2,\ldots,n}$. They present the following local algorithm that searches for vertex 1, in a graph which may be too large to hold in memory in its entirety.  
\begin{enumerate}[label=\arabic*:]
\item Initialize a list $\cL$ to contain an arbitrary node $u$ in the graph.
\item {\bf while} $\cL$ does not contain node 1 {\bf do}
\item \ \ Add a node of maximum degree in $N(\cL)$ to $\cL$ {\bf od};
\item return $\cL$
\end{enumerate}
Here for vertex set $\cL$, we let $N(\cL)=\set{w\notin \cL:\;\exists v\in \cL\ s.t.\ \set{v,w}\in E_n}$.

They show that w.h.p. the algorithm succeeds in reaching vertex 1 in $O(\log^4n)$ steps.  (We assume that an algorithm can recognize vertex 1 when it is reached.)  In \cite{BBCKL}, they also show how a local algorithm to find vertex 1 can be used to give local algorithms for some other problems. We also note that Brautbar and Kearns \cite{BK} considered local algorithms in a more general context. There the algorithm is allowed to jump to random vertices as well as crawl around the graph in the search for vertices of high degree and high clustering coefficient.

We should note that, as the maximum degree in $G_n$ is $n^{1/2-o(1)}$ w.h.p., one cannot hope to have a polylog($n$) time algorithm if we have to check the
degrees of the neighbors as we progress.  Thus the algorithm above operates on the assumption that we can find the highest-degree neighbor of a vertex in $O(1)$ time.  This would be the case, for example, if the neighborhood of a vertex is stored as a linked-list which is sorted by degrees.  In the same situation, we can also determine the $K$ highest degree neighbors of a vertex in constant time for any constant $K$, and in the present manuscript we assume such a constant-time step is possible.  In particular, in this setting, each of steps 2-7 of the following \emph{Degree Climbing Algorithm} takes constant time.

We let $d_n(v)$ denote the degree of vertex $v\in V_n$. 

{\bf Algorithm {\sc DCA}:}

The algorithm generates a sequence of vertices $v_1,v_2,\ldots,$ until vertex 1 is reached.
\begin{enumerate}[label=\textbf{Step \arabic*}]
\item Carry out a random walk on $G$ until it is mixed; i.e., until the variation distance between the current vertex and the steady state is $o(1)$.  We let $v_1$ be the terminal vertex of the walk. (See Remark \ref{remstep1} for comments on this step.) 
\item $t\gets 1$.
\item {\bf repeat}
\item \ \ Let $C_t=\set{w_1,w_2,\ldots,w_{m/2}}$ be the $m/2$ neighbors of $v_t$ of largest degree.\\ \label{Step:neighbors}
(In the case of ties for the $m/2$th largest degree, vertices will be placed randomly in $C_t$ in order to make $|C_t|=m/2$. Also $m$ is large here and we could replace $m/2$ by $\rdown{m/2}$ if $m$ is odd without affecting the analysis by very much.)
\item \ \ Choose $v_{t+1}$ randomly from $C_t$.\label{Step5}
\item \ \ $t\gets t+1$.
\item {\bf until} $d_n(v_t)\geq \frac{n^{1/2}}{\log^{1/100}n}$ ({\bf SUCCESS}) or $t>2\om\log_{4/3} n$ ({\bf FAILURE}), where $\om\to\infty$ is arbitrary. 
\item Assuming Success, starting from $v_T$, where $T$ is the value of $t$ at this point, do a random walk on the vertices of degree at least $\frac{n^{1/2}}{\log^{1/20}n}$ until vertex 1 is reached.
\end{enumerate}

\begin{remark}\label{remstep1}{\it
It is known that w.h.p. the mixing time of a random walk on $G_n$ is $O(\log n)$, see Mihail, Papadimitriou and Saberi \cite{Mihail}. So we can assume that the distribution of $v_1$ is close to the steady state $\p_v=\frac{d_n(v)}{2mn}$.
}
\end{remark}

Note that Algorithm {\sc DCA} is a local algorithm in a strong sense: the algorithm only requires access to the current vertex and its neighborhood. (Unlike the algorithm from \cite{BBCKL}, it does not need access to the neighborhood of the entire set $P_t=\set{v_1,\ldots {\vt}}$ of vertices visited so far.)  Our main result is the following:
\begin{theorem}\label{th0}
If $m$ is sufficiently large then then w.h.p. Algorithm {\sc DCA} finds vertex 1 in $G_n$ in $O(\om\log n)$ time.
\end{theorem}
{\sc DCA} is thus currently the fastest as well as the ``most local'' algorithm to find vertex 1.  We conjecture that the factor $\om$ in the running time is unnecessary.
\begin{conjecture} 
Algorithm finds vertex 1 in $G_n$ in $O(\log n)$ time, w.h.p.
\end{conjecture}

We note that w.h.p. the diameter of $G_n$ is $\sim \frac{\log n}{\log\log n}$ and so we cannot expect to improve the execution time much below $O(\log n)$.

The bulk of our proof consists of showing that the execution of Steps 2--7 requires only time $O(\om\log n)$ w.h.p. for any $\om=\om(n)\to\infty$. This analysis requires a careful accounting of conditional probabilities.  This is facilitated by the conditional model of the preferential attachment graph due to Bollob\'as and Riordan \cite{diam}.  One contribution of our paper is to recast their model in terms of sums of independent copies of the rate one exponential random variables; this will be essential to our analysis.

\subsection*{Outline of the paper}
In Section \ref{condmodel} we reformulate the construction of Bollob\'as and Riordan \cite{diam} in terms of sums of independent copies of the exponential random variable of rate one. 

Section \ref{mainloop} is the heart of the paper. The aim is to show that if ${\vt}$ is not too small, then the ratio $v_{t+1}/{\vt}$ is bounded above by 3/4 in expectation. We deduce from this that w.h.p. the main loop, Steps 2--7, only takes $O(\om\log n)$ rounds. The idea is to determine a degree bound $\D$ such that many of ${\vt}$'s left neighbors have degree at least $\D$, while only few of ${\vt}$'s right neighbors have degree at least $\D$. In this way, ${\vtt}$ is likely to be significantly smaller than ${\vt}$. 

Once we find a vertex $v_T$ of high enough degree, then we know that w.h.p. $v_T$ is not very large and lies in a small connected subgraph of vertices of high degree that contains vertex one. Then a simple argument based on the worst-case covertime of a graph suffices to show that only $o(\log n)$ more steps are required.

Our proofs will use various parameters.  For convenience, we collect here in table form a dictionary of some notations, giving a brief (and imprecise) description of the role each plays in our proof, for later reference.
\begin{center}
\def\arraystretch{2}
\begin{tabular}{rll}
{} & {\bf Definition} & {\bf Role in proof}\\
$\omega:=$ & $O(\log\log n)$& An arbitrarily chosen slowly growing fucntion.\\
$\lambda_0:=$ & $\displaystyle \frac{1}{\log^{40/m} n}$ & A (usually valid) lower bound on random variables $\eta_i$ (cf.~Section \ref{s.BR}).\\
$n_1:=$& $\log^{1/100}n$ & W.h.p. the main loop never visits $v\leq n_1$.\\
$P_t$:=&$\set{v_1,\ldots {\vt}}$& The set of vertices visited up to time $t$.\\
$\Psi$:=&$(\log\log n)^{10}$& Vertices $v>\Psi v_t$ will not be important in the search for $v_{t+1}$.\\
$L$:=&$m^{1/5}$&A large constant, significantly smaller than $m$.
\end{tabular}
\end{center}

{\bf Notation:} We write $A_n\sim B_n$ if $A_n=(1+o(1))B_n$ as $n\to\infty$.
We write $\a\lesssim \b$ in place of $\a\leq o(1)+(1+o(1))\b$.  

\section{Preliminaries}\label{condmodel}
\subsection{A different model of the preferential attachement graph}
\label{s.BR}
Bollob\'as and Riordan \cite{diam} gave an ingenious construction
equivalent to the preferential attachment graph model. We choose
$x_1,x_2,\ldots,x_{2mn}$ independently and uniformly from $[0,1]$. We
then let $\set{\ell_i,r_i}=\set{x_{2i-1},x_{2i}}$ where $\ell_i<r_i$
for $i=1,2,\ldots,mn$. We then sort the $r_i$ in increasing order
$R_1<R_2<\cdots<R_{mn}$ and let $R_0=0$. We then let 
\beq{WwI}
W_j=R_{mj}\text{ and }w_j=W_j-W_{j-1}\text{ and }I_j=(W_{j-1},W_j]
\eeq
for $j=1,2,\ldots,n$. Given this we can define $G_n$ as follows: It has vertex set $V_n=[n]$ and an edge $\set{x,y},\,x\leq y$ for each pair $\ell_i,r_i$, where $\ell_i\in I_x$ and $r_i\in I_y$.

We recast the construction of Bollob\'as and Riordan as follows: we can generate the sequence $R_1,R_2,\ldots,R_{mn}$ by letting
\beq{eq1}
R_i=\bfrac{\Upsilon_i}{\Upsilon_{mn+1}}^{1/2},
\eeq
where $\Upsilon_0=0$ and
$$\Upsilon_N=\xi_1+\xi_2+\cdots+\xi_N\text{ for }N\ge1$$ 
and where $\xi_1,\xi_2,\ldots,\xi_{mn+1}$ are independent exponential rate one random variables i.e. $\Pr(\xi_i\geq x)=e^{-x}$ for all $i$. This is because  $r_1^2,r_2^2,\ldots,r_{mn}^2$ are independent and uniform in $[0,1]$ (as they are each chosen as the maximum of two uniform points) and the order statistics of $N$ independent uniform $[0,1]$ random variables can be expressed as the ratios $\Upsilon_i/\Upsilon_{N+1}$ for $1\leq i\leq N$.

We refer to the distribution of $\Upsilon_N$ as $\erl(N)$, as it is known in the literature as the Erlang distribution. 

\subsection{Important properties}
 The advantage of our modification of the variant of the Bollob\`as and Riordan construction is that if we define
\[
\eta_i:=\xi_{(i-1)m+1}+\xi_{(i-1)m+2}+\cdots+\xi_{im},
\]
then $\eta_i$ is closely related to the size of $I_i$. It can then be used to estimate the degree of  vertex $i$. This will simplify the analysis since $\eta_i$ is simply a sum of exponentials.   

In this section, we make this claim (along with other more obscure asymptotic properties of this model) precise.  In particular, we let $\cE$ denote the event that the following properties hold for $G_n$.  In the appendix, we prove that $G_n$ has all these properties w.h.p. 
\begin{enumerate}[label=\textbf{(P\arabic*)}]
\item \label{P.ups} For $\Upsilon_{k,\ell}=\Upsilon_k-\Upsilon_\ell$, we have 
\beq{ups1}
\Upsilon_{k,\ell}\in (k-\ell)\left[1\pm
    \frac{L\th_{k,\ell}^{1/2}}{3(k-\ell)^{1/2}}\right]
\eeq
for $(k,\ell)=(mn+1,0)$ or 
$$\frac{k-\ell}{m}\in \set{\om,\om+1,\ldots,n}\text{ and }k-l\geq \begin{cases}1&l=0\\ \log^2n&k\geq \log^{30}n,l>0\\ \log^{1/300}n&0<l<k<\log^{30}n.\end{cases}$$
Here, where $n_0=\frac{\l_0^2n}{\om\log^2n}$, $\l_0= \frac{1}{\log^{20/m} n}$,
$$\th_{k,\ell}=\begin{cases}\log k&\om\leq l<k\leq \log^{30}n,\\ k^{1/2}&\om\leq k\leq n^{2/5},l=0,\\(k-\ell)^{1/2}&\log^{30}n<k\leq n^{2/5},\\ \frac{(k-\ell)^{3/2}\log n}{n^{1/2}}&n^{2/5}<k\leq n_0,\\ \frac{n}{\om^{3/2}\log^{2}n}&n_0<k.\end{cases}$$
Similarly define 
$$\th_k=\begin{cases}k^{1/2}&k\leq n^{2/5},\\ \frac{k^{3/2}\log n}{n^{1/2}}&n^{2/5}<k\leq n_0,\\ \frac{n}{\om^{3/2}\log^{2}n}&n_0<k.\end{cases}$$
\item \label{P.wsa}
$\displaystyle W_i\in \bfrac{{i}}{n}^{1/2}\left[1\pm \frac{L\th_i^{1/2}}{i^{1/2}}\right]\sim \bfrac{i}{n}^{1/2}\text{ for }\om\leq i\leq n.$
\item \label{P.wsb}
$\displaystyle w_i\in \frac{\eta_{{i}}}{2m(in)^{1/2}}\left[1\pm\frac{2L\th_i^{1/2}}{m^{1/2}i^{1/2}}\right] \sim\frac{\eta_i}{2m(in)^{1/2}}\text{ for }\om\leq i\leq n.$
\item $\displaystyle \l_0\leq \eta_i\leq 40m\log\log n\text{ for }i\in[\log^{30}n]$.\label{P.wsc}
\item $\eta_i\leq \log n$ for $i\in [n]$.\label{P.P}
\end{enumerate}
Some properties give asymptotics for intermediate quantities in the Bollobas/Riordan model (e.g., \ref{P.wsa}, \ref{P.wsb}), while the rest give worst-case bounds on parameters in various ranges for $i$.  The very technical \ref{P.ups} is just giving constraints on the gaps between the points $\Upsilon_{k}$ in the Bollobas/Riordan model.
\subsection{Inequalities}
We will use the following inequalities from Hoeffding \cite{Hoeff} at several points in the paper. Let $Z=Z_1+Z_2+\ldots+Z_N$ be the sum of independent $[0,1]$ random variables and suppose that $\m=\E(Z)$. Then if $\a>1$ we have
\begin{align}
\Pr(Z\geq (1+\a)\m)&\leq \exp\set{-\frac{\a^2\m}{2+\a/3}}\leq \begin{cases}\exp\set{-\frac{\a^2\m}{3}},&\a\leq 1.\\\exp\set{-\frac{\a\m}{3}},&\a>1.\end{cases}\label{HoHo}\\
\Pr(Z\geq \b\m)&\leq e^{-\m}\bfrac{e}{\b}^{\b\m},\;\b>1\label{HoHoHo}\\
\Pr(Z\leq (1-\a)\m)&\leq \exp\set{-\frac{\a^2\m}{2}},\quad0\leq\a\leq 1.\label{HoHoa}
\end{align}
Our main use for these inequalities is to get a bound on vertex degrees, see Section \ref{secdegree}.

In addition to these concentration inequalities, we use various inequalities bounding the tails of the random variable $\eta$. We note that the probability density $\phi(x)$ of the sum $\eta$ of $m$ independent exponential rate one random variables is given by 
\[
\phi(x):=\frac{x^{m-1}e^{-x}}{(m-1)!}.
\]
That is,
\beq{density}
\Pr(a\leq \eta\leq b)=\int_{a}^b\phi(y)dy.
\eeq
The equation \eqref{density} is a standard result, which can be verified by induction on $m$ (for example, see exercise 4.14.10 of Grimmett and Stirzaker \cite{GS}).  Although we will frequently need to bound the probability \eqref{density}, this integral cannot be evaluated exactly in general, and thus we will often use simple bounds on $\phi(\eta)$.  We summarise what we need in the following lemma:

\begin{lemma}\label{lemeta}\ 
\begin{enumerate}[label=(\alph*)]
\item 
\beq{l.etamx}
\Pr(\eta\leq xm)\leq m(xe^{1-x})^m\quad\text{for }x\leq 1-\frac{1}m.
\eeq
\item
\beq{trivial}
\Pr(\eta\leq x)\leq (1-e^{-x})^m\leq x^{m}.
\eeq
\item 
$$\Pr(\eta\geq \b m)\leq \bfrac{e\b}{e^{\b}}^m\leq e^{-3m/10}\text{ for $\b\geq 2$}.$$
\item
$$\Pr(\eta\geq (1+\a)m)\leq e^{-\a^2m/3}\text{ for $0<\a<1$}.$$
\item 
$$\Pr(\eta\leq (1-\a)m)\leq e^{-\a^2m/2}\text{ for $0<\a<1$}.$$
\end{enumerate}
\end{lemma}
\proofstart
(a) $\phi(\eta)$ is maximized at $\eta=m-1$. Taking $\phi(mx)$ $(x\leq 1-1/m)$ as an upper bound on $\phi(y)$ for $y\in [0,mx]$ and $m!\geq (m/e)^m$ in \eqref{density} gives us \eqref{l.etamx}.

(b) Writing $\eta=\xi_1+\xi_2+\cdots+\xi_m$ we have $\Pr(\eta\leq x)\leq \prod_{i=1}^m\Pr(\xi_i\leq x)$.

(c) If $\eta=\xi_1+\xi_2+\cdots+\xi_N$, then with $\l=(\b-1)/\b$,
\begin{multline}\label{mmm}
\Pr(\eta\geq \b m)=\Pr(e^{\l \eta}\geq e^{\l\b m})
\leq e^{-\l\b m}\E(e^{\l \eta})=e^{-\l\b m}\prod_{i=1}^m\E(e^{\l \xi_i})=\\
e^{-\l\b m}(1-\l)^{-m}=(\b e^{-(\b-1)})^m.
\end{multline}
(d) Putting $\b=1+\a$ into \eqref{mmm} we see that
$$\Pr(\eta\geq(1+\a)m)\leq ((1+\a)e^{-\a})^m\leq e^{-\a^2 m/3}.$$

(e) With $\l=\a/(1-\a)$ we now have
\begin{multline}\label{mmmm}
\Pr(\eta\leq (1-\a)m)=\Pr(e^{-\l \eta}\geq e^{-\l(1-\a)m})
\leq e^{\l(1-\a)m}\E(e^{-\l \eta})=e^{\l(1-\a)m}\prod_{i=1}^m\E(e^{-\l \xi_i})=\\
e^{\l(1-\a)m}(1+\l)^{-m}=((1-\a)e^{\a})^m\leq e^{-\a^2m/2}.
\end{multline}
\proofend

\subsection{Properties of the degree sequence}\label{secdegree}
We will use the following properties of the degree sequence throughout: let
\begin{align}
\z(i)&=\bfrac{n}{i}^{1/2}\brac{1-\bfrac{i}{n}^{1/2}-\frac{5L\log\log n}{\om^{3/4}\log n}}\label{zetax}\\
\z^+(i)&=\bfrac{n}{i}^{1/2}\brac{1-\bfrac{i}{n}^{1/2}+\frac{5L\log\log n}{\om^{3/4}\log n}}.\label{zeta+x}
\end{align}
Note that
\begin{align}
\z(i)&\sim\z^+(i)\quad\text{ if }i\leq n\brac{1-\frac{2\log\log n}{\log n}}.\label{largei}\\
\z(i)&\sim\bfrac{n}{i}^{1/2}\quad\text{ if }i=o(n).\label{smalli}\\
\end{align}
Also, let $\bar{d}_n(i)$ denote the expected value of $d_n(i)$ in $G_n$.
\begin{lemma}\label{lemd}\ 
\begin{enumerate}[label=(\alph*)]
\item If $\cE$ occurs then $\bar{d}_n-m\in [\eta_i\z(i),\eta_i\z^+(i)].$
\item $\Pr(d_n(i)-m\leq (1-\a)\eta_i\z(i))\leq e^{-\a^2\eta_i\z(i)/2}$ for $0\leq\a\leq 1$.
\item $\Pr(d_n(i)-m\geq (1+\a)\eta_i\z^+(i))\leq e^{-\a^2\eta_i\z^+(i)/3}$ for $0\leq\a\leq 1$.
\item $\Pr(d_n(i)-m\geq \b \eta_i\z^+(i))\leq (e/\b)^{\b\eta_i\z^+(i)}$ for $\b\geq 2$.
\item W.h.p. $\eta_i\geq \l_0$ and $\om\leq i\leq n^{1/2}$ implies that $d_n(i)\sim\eta_i\bfrac{n}{{i}}^{1/2}$.
\item W.h.p. $\om\leq i\leq \log^{30}n$ implies that $d_n(i)\sim\eta_i\bfrac{n}{{i}}^{1/2}$.
\item  W.h.p. $\om\leq i\leq n^{1/2} \text{ implies } d_n(i)\lesssim \max\set{1,\eta_i}\bfrac{n}{{i}}^{1/2}.$
\item  W.h.p. $n^{1/2}\leq i\leq n \text{ implies } d_n(i)\leq n^{1/3}.$
\item W.h.p. $1\leq i\leq \log^{1/49}n$ implies that $d_n(i)\geq \frac{n^{1/2}}{\log^{1/20}n}$.
\item W.h.p. $d_n(i)\geq \frac{n}{\log^{1/20}n} \text{ implies } i\leq \log^{1/9}n$.
\end{enumerate}
\end{lemma}
\proofstart
We defer the proof, which is straightforward but tedious, to the appendix.
\proofend

\begin{remark}\label{remcond}
We will for the rest of the paper condition on the occurrence of $\cE$. All probabilities include this conditioning. We will omit the conditioning in the text in order to simplify expressions.
\end{remark}

\section{Analysis of the main loop}\label{mainloop}
Since the variation distance after Step 1 is $o(1)$, it suffices to prove Theorem \ref{th0} under the assumption that we begin Step 2, with $v_1$ chosen randomly, exactly according to the stationary distribution.

The main loop consists of Steps 2--7. Let $v_0=1$ and $v_1,v_2,\ldots,v_s$
for $s\geq 1$ be the sequence of vertices followed by the algorithm up to time $s$.  Let $\rho_t=v_{t+1}/{\vt}$, and define $T_1,T_2$ by
\beq{T1}
T_1=\min\set{t:{\vt}\leq \log^{30}n}\text{ and }T_2=T_1+30\om\log_{4/3}\log n\text{ and }T_0=\min\set{2\omega\log_{4/3}n,T_2}.
\eeq
We will prove, see Lemma \ref{ratio1}, that 
\beq{neq0}
\E(\rho_t\ff)\leq \frac34\text{ for }1\leq t\leq T_0.
\eeq 
Recalling that $T$ is the time when Step 8 begins, we note that if $T<t\leq T_0$ then this statement is meaningless. So, we will keep to the following notational convention: if $X_t$ is some quantity that depends on $t\leq T$ and $t>T$ then $X_t=X_T$.

Now, roughly speaking, if $r=2\log_{4/3}n$ and $\m$ is the number of steps in the main loop, then we would hope to have 
\beq{hope}
\Pr(\m\geq r)\leq \Pr\brac{\r_0\rho_1\cdots \rho_r\geq \frac1{n}}\leq n\E(\rho_0\rho_1\cdots \rho_r)\leq \frac1n
\eeq
and so w.h.p. the algorithm will complete the main loop within $2\log_{4/3}n$ steps. Unfortunately, we cannot justify the last inequality, seeing as the $\r_t$ are not independent. I.e. we cannot replace $\E(\rho_0\rho_1\cdots \rho_r)$ by $\prod_{i=0}^r\E(\r_i)$. We proceed instead as in the next lemma.
\begin{lemma}\label{MAIN}
Assuming \eqref{neq0} we have the w.h.p. {\sc DCA} completes the main loop in at most $T_0$ steps with SUCCESS.
\end{lemma}
\proofstart
We let $s_0$ denote the number of vertices visited by the main loop, and then define $Z_s=\rho_0\r_1\cdots\r_s$ for $s\leq s_0$, and $Z_s=\rho_0\r_1\cdots \rho_{s_0} (\tfrac 34)^{s-s_0}$ for $s>s_0$.

Suppose first that $T_1>\om\log_{4/3}n$. Now \eqref{neq0} and Jensen's inequality implies that for $s\geq 1$,
\begin{multline}\label{L1}
\E(\log(Z_s)\ff)=\sum_{i=0}^{\min(s,s_0)}\E(\log(\r_i)\ff)+\sum_{\min(s,s_0)+1}^s \log\tfrac34\\
\leq\sum_{i=0}^{\min(s,s_0)}\log\E(\r_i\ff)+\sum_{\min(s,s_0)+1}^s \log \tfrac 34\leq s\log(3/4).
\end{multline}
Now 
\beq{L2}
\log(Z_s)\geq (s-s_0)\log(3/4)-\log n\geq s\log(3/4)-\log n
\eeq
since $\r_1\r_2\dots \r_{s_0}\geq 1/n$.

Now let 
$$\a=\Pr(\log(Z_s)\leq (1-\b)s\log(3/4)\ff)$$ 
where $\a,\b$ are to be determined. Then, \eqref{L1}, \eqref{L2} imply that
\beq{L3}
(1-\a)(1-\b)s\log(3/4)+\a(s\log(3/4)-\log n)\leq \E(\log(Z_s)\ff)\leq s\log(3/4).
\eeq
Equation \eqref{L3} then implies that
\beq{L4}
\a\geq \frac{\b s\log(4/3)}{\b s\log(4/3)+\log n}.
\eeq
Now putting $s=\om\log_{4/3}n$ and $\b=1/2$ we see that \eqref{L4} becomes
\beq{L5}
\a\geq 1-\frac{2}{\om+2}=1-o(1).
\eeq
So w.h.p. after at most $\om\log_{4/3}n$ steps, we will have exited the main loop with SUCCESS.

Suppose now that $T_1\leq \om\log_{4/3}n$. Using the argument that gave us \eqref{L4} we obtain 
\beq{small}
T-T_1\leq \om\log_{4/3}\log^{30}n\ w.h.p.
\eeq 
\proofend

To prove Lemma \ref{ratio1}, we will use a method of deferred decisions, exposing various parameters of $G_n$ as we proceed.  At time $t$, we will consider all random variables in the model from Section \ref{s.BR} as being exposed if they have affected the algorithm's trajectory thus far, and condition on their particular evaluation.  To reduce the conditioning necessary, we will actually analyze a modified algorithm, \textbf{NARROW-DCA}$(\tau)$, and then later show that the trajectory of \textbf{NARROW-DCA}($\tau$) is the same as that of the \textbf{DCA} algorithm, w.h.p., when identical sources of randomness are used.

\textbf{NARROW-DCA}$(\tau)$ is the same as the DCA algorithm, except that for the first $\tau$ rounds of the algorithm, a modified version of \ref{Step:neighbors} is used:\\
\textbf{Modified \ref{Step:neighbors}}  Let 
\beq{Ct}
C_t=\set{w_1,w_2,\ldots,w_{m/2}}
\eeq 
be the $m/2$ neighbors of $\vt$ of largest degree from $\{1,\dots,\Psi{\vt}\}$ where $\Psi:=(\log\log n)^{10}$.

For rounds $\tau+1,\tau+2,\dots,$ the behavior of NARROW-DCA$(\tau)$ is the same as for DCA.

\smallskip
Notice that \textbf{NARROW-DCA} ``cheats'' by using the indices of the vertices, which we do not actually expect to be able to use.  Nevertheless, we will see later that w.h.p., for $\t=2\om\log_{4/3}n$, the path of this algorithm is the same as for the DCA algorithm, justifying its role in our analysis.

\subsection{Analyzing one step}
Our analysis of one step of the main loop consists of the following lemma:
\begin{lemma}\label{ratio1}
Let $\r_t$ be the ratio of $\vtt/\vt$ which appears in a run of the algorithm \textup{\textbf{NARROW-DCA}}$(t)$.  Then for all $t\leq T_0$ (see \eqref{T1}), we have that
\beq{ratio}
\E(\r_t\ff)\leq \frac{3}{4}\quad\text{and}\quad\Pr\brac{\r_t\geq \Psi}\leq \frac{1}{\log^2 n}.
\eeq
\end{lemma}
The first statement ensures that \textbf{NARROW-DCA}$(t)$ makes progress in expectation in the $t$th jump.  The second part of this statement implies by induction that for any $t\leq \omega\log n$, the behavior of  \textbf{NARROW-DCA}$(t)$ is identical to the behavior of the DCA algorithm for the first $t$ steps.  Thus together these statements give \eqref{neq0}.

To prove Lemma \ref{ratio1}, we will prove a stronger statement which is conditioned on the \emph{history} of the algorithm at time $t$.  The history $\cH_t$ of the process at the \emph{end} of step $t$ consists of 
\begin{enumerate}[label=\textbf{(H\arabic*)}]
\item The sequence $v_1,v_2,\ldots,{\vt}$.\label{H.1}
\item The left-choices $\l(v_s,1),\l(v_s,2),\ldots,\l(v_s,m),1\leq s<t$ and the corresponding left neighbors $N_L(v_s)=\set{u_{1,s},u_{2,s},\ldots,u_{m,s}}$. These are the $m$ $\ell_i'$s that correspond to the $m$ $r_i$'s associated with $v_s$ as defined at the beginning of Section \ref{s.BR}. \label{H.1a}
\item The lists $u'_{1,s},u'_{2,s},\ldots,u'_{r,s}$ of all vertices $u'_{k,s}$ which have the property that (i) $v_s\in N_L(u_{k,s}')$ and (ii) $u'_{k,s}\leq \Psi v_s$ for $1\leq s<t$. (It is important to notice that $s<t$ here.)\label{H.1b}
\item The values $\eta_{v_i}$ and the intervals $I_{v_i}$ for $i=1,2,\ldots,t$.\label{H.2}
\item The values $\eta_w$ and the intervals $I_w$ and the degrees  $\deg(w)$, for $w\in \bigcup_{i=1}^tN(v_i)$.
\end{enumerate}
Here, 
$$N(v)=N_L(v)\cup N_R(v)\text{ where }N_R(v)=\set{w\leq \Psi v:v\in N_L(w)}.$$

We note that at any step $t$, and for a fixed random sequence used in the \textbf{NARROW-DCA}$(t)$ algorithm, $\cH_t$ contains all random variables which have determined the behavior of the algorithm so far, in the sense that if we modify any random variables from the random graph model described in Section \ref{s.BR} while preserving all values in the history, then the trajectory of the algorithm will not change.   We write $H_t$ to refer to a particular evaluation of the history (so that we will be conditioning on events of the form $\cH_t=H_t$).

\subsubsection*{Structure of the proof}
The essential structure of our proof of Lemma \ref{ratio1} is as follows:
\begin{enumerate}[label=\textbf{Part \arabic*}]
\item \label{step:typical} We will define the notion of a \emph{typical} history $H_t$.
\item \label{step:htcond} We will prove that for $t\leq T_0$ and any typical history $H_t$, random variables $\eta_v$ which are not explicitly exposed in $H_t$ are essentially unconditioned by the event $\cH_t=H_t$ (Lemma \ref{l.cond}).
\item \label{step:htind} We will prove by induction that $\cH_t$ is typical w.h.p., for $t\leq T_0$.\\
\item \label{step:23eta} We will use \ref{step:htcond} and \ref{step:htind} to prove that for $t\leq T_0$, 
\beq{bk1}
\E(\r_t\mid H_t)\leq \frac{2}{3}+\frac{21\eta_{{\vt}}}{mL}+\frac{L^3}{m^2} \quad\text{and}\quad \Pr\brac{\r_t\geq \Psi\mid H_t}\leq \frac{1}{\log^2 n}
\eeq
by using using nearly unconditioned distributions of random variables which are not revealed in $H_t$ to estimate the probabilities of various events.  Here $\E(\r_t\mid H_t)$ is short for $\E(\r_t\mid \cH_t=H_t)$. (Note that $\eta_{\vt}$ in \eqref{bk1} is simply a real number determined by $H_t$.) In this context, we always work under the assumption that $H_t$ is typical.
\item \label{step:4m} We will also prove for $t\leq T_0$ that 
\beq{bk2}
\E(\eta_{{\vtt}})\leq 4m.
\eeq
\end{enumerate}
Now the expected value statement in \eqref{ratio} follows from \eqref{bk2} and the first part of \eqref{bk1}, by removing the conditioning on $H_t$.  
\subsection*{\ref{step:typical}}  Let $P_{t}$ denote the sequence of vertices $v_1,v_2,\dots,v_{t}$ determined by the history $H_t$. We now define the notion of a \emph{typical} history $H_t$.  For this purpose, we consider the reordered values $0\leq \lal{t}{1}<\lal{t}{2}<\cdots<\lal{t}{N(t)}$ where  
$$\La_0^{(t)}=\set{\lal{t}{1},\lal{t}{2},\ldots,\lal{t}{N(t)}}=\set{\l(v_s,i):1\leq s\leq t,\,1\leq i\leq m}.$$
Given this we define $v=\vv{t}{j}$ to be the index such that $\lal{t}{j}\in I_v$ and then let $$V_L^{(t)}=\set{\vv{t}{j}:1\leq j\leq N(t)}.$$ 
{
We also define
$$V_R^{(t)}=\set{v:v\in N_R(P_t)}$$
Now let us reorder 
$$V^{(t)}=\set{\ww{t}{1}<\ww{t}{2}<\cdots<\ww{t}{M(t)}}= V_L^{(t)}\cup V_R^{(t)}.$$

We define the extreme points $\ww{t}{0}=0$ and $\ww{t}{M(t)+1}=n+1$ and define
\begin{align}
\WW{t}{j}&=[\ww{t}{j-1}+1,\ww{t}{j}-1]\quad \text{and}\quad X^{(t)}=\bigcup_{j=1}^{M(t)+1}\WW{t}{j}=[n]\setminus V^{(t)}\quad\text{and}\quad N_j^{(t)}=|\WW{t}{j}|,\\
\UU{t}{j}&=[W_{\ww{t}{j-1}+1},W_{\ww{t}{j}-1}]\quad\text{and}\quad
U^{(t)}=\bigcup_{j=1}^N\UU{t}{j}\quad\text{and}\quad L^{(t)}_j=|\UU{t}{j}|.
\end{align}
A typical history $H_t,t\leq T_0$ is now one with the following properties:
\begin{enumerate}[label=\textbf{(S\arabic*)}]
\item There do not exist $s_1,s_2\leq t$ such that either (i) $s_1\leq t-2$ and $v_{s_1}$ and $v_{s_2}$ are neighbors or (ii) $s_1\leq t-3$ and there exists a vertex $w$ such that $w\in N(v_{s_1})\cap N(v_{s_2})$. (We say that the path is \emph{self-avoiding}.)\label{S1}
\item \label{Wsep} The points of $\La^{(t)}$ are {\em well-separated}, in the following sense: 
\beq{Gj}
|\ww{t}{j}-\ww{t}{j-1}|\geq \begin{cases}\log^2n&\ww{t}{j-1}\geq \log^{30}n.\\\log^{1/400}n&{\rm Otherwise}.\end{cases}
\eeq
\end{enumerate}
We observe that 
\begin{enumerate}[label=\textbf{(T\arabic*)}]
\item If $H_t$ is typical then $v_{j+1}$ is chosen from $X^{(j)}$ for all $j<t$.
\item Each $\UU{t}{j}$ is the union of intervals $I_v,v\in \WW{t}{j}$.
\end{enumerate}

\subsection*{\ref{step:htcond}}  We prove the following:
\begin{lemma}\label{l.cond}
For any vertex $v\in X^{(t)}$, any interval $R\subseteq \Re$, and any typical history $H_t$, we have that $v\notin P_t\cup N(P_t)$ implies
\beq{uneta}
\Pr\brac{\eta_v\in R\mid H_t}\sim\Pr\brac{\erl(m)\in R}.
\eeq
\end{lemma}
The following lemma is the starting point for the proof of Lemma \ref{l.cond}.
\begin{lemma}\label{unbias}
Let $j\in [M(t)+1]$, let $H_t$ be any typical history, and let $X'$ be the value of $\WW t j$ in $H_t$.   Then the distribution of the random variables $\eta_v,v\in X'$ conditioned on $\cH_t=H_t$ is equivalent to the distribution of the random variables $\eta_v$, $v\in X'$ conditioned only on the relationship $\sum_{v\in X'} \eta_v=A_1^2-A_0^2$, where $A_1,A_0$ are the values of $W_{\ww{t}{j}-1}$ and $W_{\ww{t}{j-1}+1}$, respectively, in $H_t$.
\end{lemma}
\proofstart
Suppose we fix everything except for $\eta_v,v\in X'$. By everything we mean every other $\eta_w$ and all of the $\l(v,i)$ and the random bits we use to make our choices in \ref{Step5} of DCA; we let $H_t$ be the corresponding history. Suppose now that we replace $\eta_v,v\in X'$ with $\eta_v',v\in X'$ without changing the sum $\sum_{v\in X'} \eta_v$.  Then $W_{\ww{t}{j-1}+1}$ remains the same, as it depends only on $\eta_v$ for $v\notin X'$, and thus $W_{\ww{t}{j}-1}$ remains the same as well, since the difference $A_1^2-A_0^2$ is unchanged.

In particular, this implies that $H_t$ remains a valid history. We confirm this by induction. Suppose that $H_{s},s<t$ remains valid. We first note that because the $\l(v_s,i)$ are unchanged, none of $v_s's$ left neighbors are in $\WW{t}{j}$. Also, $N_R(v_s)$ and the vertex degrees for $w\in N_R(v_s)$ will not be affected by the change, even if $v_s<\min\WW{t}{j}$. So $H_{s+1}$ will be unchanged, completing the induction.
\proofend
We are now ready to prove Lemma \ref{l.cond}.
\begin{proof}[Proof of Lemma \ref{l.cond}]
Suppose that $v\in X'=\WW{t}{j}$, then $M=N_j^{(t)}\geq\z_n\to\infty$.
We now use Lemma \ref{unbias} to write
\[
\Pr(\eta_v\leq x\mid H_t)=\Pr\brac{\eta_v\leq x\bigg| \sum_{w\in X'}\eta_w=A_1^2-A_0^2,}
\]
where $A_1$ and $A_0$ are the values of $W_{\ww{t}{j}-1}$ and $W_{\ww{t}{j-1)+1}}$, respectively, in $H_t$, so that $A_1-A_0$ is the value of $L_j^{(t)}$ in $H_t$.

Now from \ref{P.ups} we have that $A:=A_1^2-A_2^2\in [(1-\e)mM,(1+\e) mM]$ for $M=|X'|$ w.h.p., for any $\e>0$. Thus we fix any $\mu\in [(1-\e)mM,(1+\e)mM]$ and show that 
\[
\Pr\brac{\eta_v\leq x\bigg| \sum_{w\in X'}\eta_w=\mu}=(1+O(\e))\Pr\brac{\erl(m)\leq x}.
\]
The lemma follows since $\e$ is arbitrary.

We write
\begin{align}
\label{o(1)}
&\Pr\brac{\eta_v\leq x\bigg| \sum_{w\in X'}\eta_w=\mu}\\
&=\int_{\eta=0}^{x}\frac{\eta^{m-1}e^{-\eta}}{(m-1)!}\cdot\frac{(\mu-\eta)^{(M-1)m-1}e^{-(\mu-\eta)}}{((M-1)m-1)!} \cdot\frac{(Mm-1)!}{\mu^{Mm-1}e^{-\mu}}d\eta\\
&=\int_{\eta=0}^{x}\frac{\eta^{m-1}e^{-\eta}}{(m-1)!}\cdot \frac{\brac{1-\frac{\eta}{\mu}}^{(M-1)m-1}e^\eta\prod_{i=1}^m(Mm-i)}{\mu^m}d\eta\\
&=\int_{\eta=0}^{x}\frac{\eta^{m-1}e^{-\eta}}{(m-1)!}\cdot\exp\set{\eta-((M-1)m-1)\brac{\frac{\eta}{\mu}+O\bfrac{\eta^2}{\mu^2}}}\cdot \brac{1+O\bfrac{m}{M}}d\eta\\
&=(1+O(\e))\int_{\eta=0}^{x} \frac{\eta^{m-1}e^{-\eta}}{(m-1)!}d\eta.
\end{align}
Here we used that $H_t$ typical implies that $M\geq \log^{1/400} n\to \infty$.
\end{proof}

\subsection*{\ref{step:htind}}
In the next section we will need a lower bound on $v_{t+1}$.
Let
$$\f_v=\begin{cases}\frac{1}{\log^3n}&v\geq \log^{30}n\\\frac{1}{(\log\log n)^3}.&v<\log^{30}n.\end{cases}$$
\begin{lemma}\label{phiv}
W.h.p. $\r_t\geq \f_{v_t}$ for $1\leq t\leq T_0$.
\end{lemma}
\proofstart
The values of $\l(v_t,i),i=1,2,\ldots,m$ are unconditioned by $H_t$, see \ref{H.1a}. It then follows from \ref{P.wsa} that if $v_t\geq \log^{30}n$ then 
\beq{fv}
\Pr(v_{t+1}\leq \f_{v_t}v_t\mid H_t)\lesssim m\frac{W_{\f_{v_t}}}{W_{v_t}}\lesssim m\f_{v_t}^{1/2}=\frac{m}{\log^{3/2}n}.
\eeq
There are $O(\om\log n)$ choices for $t$ and so this deals with $v_t\geq \log^{30}n$. 

Now there are $O(\log\log n)$ choices of $t\in[T_1,T_0]$ for which $v_t\leq \log^{30}n$. In this case we can replace the RHS of \eqref{fv} by $1/(\log\log n)^{3/2}$.
\proofend

We will also need to bound the size of $N_R(v_t)$ for all $t$. 
\begin{lemma}\label{NRsize}
W.h.p., for all $t\leq T_0$,
$$|N_R(v_t)|\leq \begin{cases}\log^3n&v_t\geq \log^{30}n.\\ (\log\log n)^{20}&v_t\leq \log^{30}n.\end{cases}$$
\end{lemma}
\proofstart
The size of $N_R(v), v=v_t$ is stochastically bounded by $Bin(\Psi v,\eta_v/v)$. This is because if $w\in N_R(v)$ then $w\leq \Psi v$. Also, for any such $w$, the probability that it has $v$ as a left neighbor is at most $mw_v/W_w\lesssim \eta_v/(vw)^{1/2}\leq \eta_v/v$. This uses property \ref{S1} to see that the values of $\l(w,i),i=1,2,\ldots,m$ are unconditioned by $H_t$. Thus, if $\th_v=\log^3n$ if $v\geq \log^{30}n$ and equal to $(\log\log n)^{20}$ otherwise,
\beq{NRsz}
\Pr(|N_R(v)|\geq \th_v\mid H_t)\leq \binom{\Psi v}{\th_v}\bfrac{\eta_v}{v}^{\th_v}\leq \bfrac{e\Psi \eta_v}{\th_v}^{\th_v}.
\eeq
If $v\geq \log^{30}n$ then the RHS of \eqref{NRsz} is at most $(e/\log n)^{\log^3n}$ which is clearly small enough to handle $T$ possible values for $t$. If $v<\log^{30}n$ then the RHS of \eqref{NRsz} is at most $(40e/(\log\log n)^9)^{(\log\log n)^{20}}$ which is small enough to handle $O(\om\log\log n)$ possible values for $t$ such that $v<\log^{30}n$.
\proofend

Continuing \ref{step:htind}, we now show that the DCA walk doesn't contain cycles.
\begin{lemma}\label{DCApath}
W.h.p. the path $P_t,t\leq T_0$ is self avoiding. 
\end{lemma}
\proofstart
We proceed by induction and assume that the claim of the lemma is valid up to time $t-1$. Now consider the choice of $v_t$.

{\bf Case 1: There is an edge $v_sv_t$ where $s\leq t-2$:} \\
{\bf (a):} $v_t\in N_L(v_s)\cap N_L(v_{t-1})$.\\
We bound the probability of this (conditional on $\cE,H_t$) asymptotically by
\beq{cs1a}
\sum_{s\in[t-2]}\sum_{v\in N_L(v_s)}\frac{mw_v}{W_{v_{t-1}}}\lesssim\sum_{s\in[t-2]}\sum_{v\in N_L(v_s)}\frac{\eta_{v}}{2(vv_{t-1})^{1/2}}.
\eeq
Here, and throughout the proof of Case 1, $v$ denotes a possibility for $v_t$ and $mw_v/W_{v_{t-1}}$ bounds the probability that $v_{t-1}$ chooses $v$. Remember that these choices are still uniform, given the history. 

We split the sum in \eqref{cs1a} as
\beq{cs1}
\sum_{\substack{s\in[t-2]\\v_s>\log^{30} n}}\sum_{v\in N_L(v_s)}\frac{\eta_{v}}{2(vv_{t-1})^{1/2}}+
\sum_{\substack{s\in[t-2]\\v_s\leq \log^{30} n}}\sum_{v\in N_L(v_s)}\frac{\eta_{v}}{2(vv_{t-1})^{1/2}}.
\eeq
Consider the first sum. There are less than $t$ choices for $s$; $m$ choices for $v$ and $\eta_v\leq \log n$. Now $v\in N_L(v_s)$ and Lemma \ref{phiv} implies that $v\geq \log^{27}n$. So we can bound the first sum by 
$$(\#\ of\ s)\cdot(\#\ of\ v)\cdot(\max\eta_v)\cdot\frac{1}{v^{1/2}}\leq T_0\cdot m\cdot\log n\cdot\frac{1}{\log^{27/2}n}=o\bfrac{1}{\log^{11}n}.$$ 
Summing this estimate over $t\leq T_0$ gives $o(1)$. 

For the second sum, we bound the number of choices of $s$ by $O(\om\log\log n)$ and $\eta_v$ by $O(\log\log n)$, since $v\leq v_s$. We use the fact (see Section \ref{exiting}) that $v_{t-1}\geq \log^{1/100}n$. So we can therefore bound the second sum by 
$$(\#\ of\ s)\cdot(\#\ of\ v)\cdot(\max\eta_v)\cdot\frac{1}{v_{t-1}^{1/2}}\leb \om\log\log n\cdot m\cdot\log\log n\cdot\frac{1}{\log^{1/200}n}=o\bfrac{1}{\log^{1/300}n}.$$ 
(We use $A\leb B$ in place of $A=O(B)$.)\\
There are $O(\om\log\log n)$ choices for $T_0\geq t>s\geq T_1$ and so we can sum this estimate over choices of $t$.

{(b):} $v_t\in N_L(v_s)\cap N_R(v_{t-1})$.\\
Using $v_{t}\in N_R(v_{t-1})$, we bound the probability of this asymptotically by
\beq{cs2}
\sum_{\substack{s\in[t-2]\\v_s>\log^{30} n}}\sum_{v\in N_L(v_s)}\frac{\eta_{v_{t-1}}}{2(vv_{t-1})^{1/2}}+
\sum_{\substack{s\in[t-2]\\v_s\leq \log^{30} n}}\sum_{v\in N_L(v_s)}\frac{\eta_{v_{t-1}}}{2(vv_{t-1})^{1/2}}.
\eeq
For the first sum we use the argument of Case (a) without any change, except for bounding $\eta_{v_{t-1}}$ by $\log n$ as opposed to bounding $\eta_v$ by the same. This gives a bound
\beq{bbb}
(\#\ of\ s)\cdot(\#\ of\ v)\cdot(\max\eta_{v_{t-1}})\cdot\frac{1}{v^{1/2}}\leb T_0\cdot m\cdot\log n\cdot\frac{1}{\log^{27/2}n}=o\bfrac{1}{\log^{11}n}.
\eeq
This is small enough to inflate by the number of choices for $t$.

For the second sum we split into two cases: (i) $v_{t-1}\geq \log^{30}n$ and (ii) $v_{t-1}< \log^{30}n$. This enables us to control $\eta_{v_{t-1}}$. For the first case we obtain
$$(\#\ of\ s)\cdot(\#\ of\ v)\cdot(\max\eta_{v_{t-1}})\cdot\frac{1}{v_{t-1}^{1/2}}\leq \om\log\log n\cdot m\cdot\log n\cdot\frac{1}{\log^{15}n}=o\bfrac{1}{\log^{13}n}.$$
The RHS is small enough to handle the $O(\om\log n)$ choices for $t$.

For the second case we obtain
$$(\#\ of\ s)\cdot(\#\ of\ v)\cdot(\max\eta_{v_{t-1}})\cdot\frac{1}{v_{t-1}^{1/2}}\leq \om\log\log n\cdot m\cdot\log\log n\cdot\frac{1}{\log^{1/200}n}=o\bfrac{1}{\log^{1/300}n}.$$
The RHS is small enough to handle the $O(\om\log\log n)$ choices for $t$.

{(c):} $v_t\in N_R(v_s)\cap N_L(v_{t-1})$.\\
Using $v_t\in N_L(v_{t-1})$, we bound the probability of this asymptotically by
\beq{cs3}
\sum_{\substack{s\in[t-2]\\v_s>\log^{29} n}}\sum_{v\in N_R(v_s)}\frac{\eta_{v}}{2(vv_{t-1})^{1/2}}+
\sum_{\substack{s\in[t-2]\\v_s\leq \log^{29} n}}\sum_{v\in N_R(v_s)}\frac{\eta_{v}}{2(vv_{t-1})^{1/2}}.
\eeq
For the first sum we use $v\geq v_s$ and the argument of Case (a) without change, but notice we split over $v_s>\log^{29}n$ or not here. This gives a bound of
$$(\#\ of\ s)\cdot(\#\ of\ v)\cdot(\max\eta_{v})\cdot\frac{1}{v^{1/2}}\leq T_0\cdot m\cdot\log n\cdot\frac{1}{\log^{29/2}n}=o\bfrac{1}{\log^{12}n}.$$

For the second sum we use $v\leq \Psi v_s$ to bound $v$ by $\log^{30}n$. We also use Lemma \ref{NRsize} to bound the number of choices of $v$ by $(\log\log n)^{20}$. This gives a bound of
$$(\#\ of\ s)\cdot(\#\ of\ v)\cdot(\max\eta_{v})\cdot\frac{1}{v_{t-1}^{1/2}}\leb \om\log\log n\cdot(\log\log n)^{20}\cdot\log\log n \cdot\frac{1}{\log^{1/200}n}=o\bfrac{1}{\log^{1/300}n}.$$

{(d):} $v_t\in N_R(v_s)\cap N_R(v_{t-1})$.\\
Using $v_{t}\in N_R(v_{t-1})$, we bound the probability of this asymptotically by
\beq{cs4}
\sum_{\substack{s\in[t-2]\\v_s>\log^{30} n}}\sum_{v\in N_R(v_s)}\frac{\eta_{v_{t-1}}}{2(vv_{t-1})^{1/2}}+
\sum_{\substack{s\in[t-2]\\v_s\leq \log^{30} n}}\sum_{v\in N_R(v_s)}\frac{\eta_{v_{t-1}}}{2(vv_{t-1})^{1/2}}.
\eeq
For the first sum we use $v\geq v_s$ and Lemma \ref{NRsize} to bound the number of choices for $v$ and then we have a bound of 
$$(\#\ of\ s)\cdot(\#\ of\ v)\cdot(\max\eta_{v_{t-1}})\cdot\frac{1}{v^{1/2}}\leq T_0\cdot \log^3n\cdot\log n\cdot\frac{1}{\log^{15}n}=o\bfrac{1}{\log^{9}n}.$$
For the second sum we split into two cases: (i) $v_{t-1}\geq \log^{30}n$ and (ii) $v_{t-1}< \log^{30}n$. This enables us to control $\eta_{v_{t-1}}$. We also use Lemma \ref{NRsize} to bound the number of choices for $v$ in each case. Thus in the first case we have the bound
$$(\#\ of\ s)\cdot(\#\ of\ v)\cdot(\max\eta_{v_{t-1}})\cdot\frac{1}{v_{t-1}^{1/2}}\leb \om\log\log n\cdot\log^3n\cdot\log n\cdot \frac{1}{\log^{15}n}=o\bfrac{1}{\log^{10}n}.$$
In the second case we have
$$(\#\ of\ s)\cdot(\#\ of\ v)\cdot(\max\eta_{v_{t-1}})\cdot\frac{1}{v_{t-1}^{1/2}}\leb \om\log\log n\cdot(\log\log n)^{20}\cdot\log\log n\cdot \frac{1}{\log^{1/200}n}=o\bfrac{1}{\log^{1/300}n}.$$
{\bf Case 2: There is a path $v_s,v,v_t$ where $s<t$.}\\
The calculations that we have done for Case 1 carry through unchanged. We just replace $v_{t-1}$ by $v_t$ throughout the calculation and treat $v$ as an arbitrary vertex as opposed to a choice of $v_{t}$. 
\proofend
\subsubsection*{The $\ww{t}{j}$ are separated}
We now prove that w.h.p. points $\l_i$ are well-separated. Let 
$$J_1=\set{j:v_j\ge\log^{30}n}.$$
\begin{lemma}
Equation \eqref{Gj} holds w.h.p. for all $t\leq T_0$.
\end{lemma}
\proofstart We consider cases.\\
{\bf Case 1: $\ww{t}{j-1},\ww{t}{j}\in V_R^{(t)}$.}\\
For this we write
$$\z_{v,w}=\begin{cases}\log^2n&\min\set{v,w}\geq \log^{30}n.\\\log^{1/300}n&{\rm Otherwise}.\end{cases}.$$
\begin{align}
&\Pr(\exists 1\leq s\leq t,v\in N_R(v_s),w\in N_R(v_t):|v-w|\leq\z_{v,w}\mid\cE,H_t)\leq\\
&\lesssim \sum_{1\leq s\leq t\leq T_0}\sum_{\substack{v\in N_R(v_s),w\in N_R(v_t)\\|v-w|\leq\z_{v,w}}} \frac{\eta_{v_s}\eta_{v_t}}{(v_sv_tvw)^{1/2}}\\
&\leq \sum_{1\leq s\leq t\leq T_0}\sum_{v\in N_R(v_s)} \frac{\z_{v,w}\eta_{v_s}\eta_{v_t}}{(v-\z_{v,w})(v_sv_t)^{1/2}}\label{R1a}\\
&\leq 2\sum_{1\leq s\leq t\leq T_0} \frac{\z^*_{s,t}\eta_{v_s}\eta_{v_t}|N_R(v_s)|}{(v_sv_t)^{1/2}}.\label{R1}
\end{align}
Here $\z^*_{s,t}$ will be a bound on the possible value of $\z_{v,w}$ in \eqref{R1a}.

{\bf Case 1a: $\max\set{v_s,v_t}\geq \log^{29}n$:}\\ 
In this case $\z^*_{s,t}\leq \log^2n$ and we can bound the summand of \eqref{R1} by $$\z^*_{s,t}\cdot\log^2n\cdot\log^3n\cdot\frac{1}{\log^{29/2}n} =\frac{1}{\log^{15/2}n}.$$ 
Multiplying by a bound $T_0^2$ on the number of summands gives a bound of $o(1)$. Here, and in the next case, we use Lemma \ref{NRsize} to bound $|N_R(v_s)|$.

{\bf Case 1b: $\max\set{v_s,v_t}< \log^{29}n$:}\\
Here we have $\max\set{v,w}\leq \Psi\log^{29}n\leq \log^{30}n$. In this case we can bound the summand of \eqref{R1} by 
$$\log^{1/300}n\cdot(40m\log\log n)^2\cdot(\log\log n)^{20}\cdot\frac{1}{\log^{1/100}n} =o\bfrac{(\log\log n)^5}{\log^{1/200}}.$$
We only have to inflate this by $(T_0-T_1)^2=O((\om\log\log n)^2)$. 
This completes the case where $\ww{t}{j-1},\ww{t}{j}\in R^{(t)}$.

{\bf Case 2: $\ww{t}{j-1},\ww{t}{j}\in V_L^{(t)}$}\\
We first show that the gaps $\l_j-\l_{j-1}$ are large. Define
\beq{defbeta}
\b_1=\frac{\log^{15/2}n}{n^{1/2}}\text{ and }\b_2=\frac{\log^{1/300}n}{n^{1/2}}.
\eeq
and 
$$\e_j=\begin{cases}\b_1&\l_j=\l(v_t,i),v_t\in J_1.\\ \b_2&otherwise. \end{cases}$$
and
$$\s_1=\frac{1}{\log^{15/2}n}\text{ and }\s_2=\frac{1}{\log^{1/200}n}.$$
We drop the superscript $t$ for the rest of the lemma. 
\begin{claim}\label{cl1}
$$\Pr(\exists \l_j\in \La_0:\l_{j-1}>\l_j-\e_j\mid H_t)=o(1).$$
\end{claim}
\begin{proof}[Proof of Claim \ref{cl1}] 
This follows from the fact that
\beq{lagap}
\Pr(\exists j:\l_{j-1}>\l_j-\e_j)\leq o(1)+(1+o(1))(m^2T_1^2\s_1+m^2(T_0-T_1)(T_1\s_1+(T_0-T_1)\s_2)).
\eeq
We have fewer than $m^2T_1^2$ choices for $s=\t(j-1),t=\t(j)\in J_1$. Assume first that $s<t$. Given such a choice, we have that w.h.p. $W_{v_t}\gtrsim\log^{15}n/n^{1/2}$ by \ref{P.wsa}. Now $\l_j$ will have been chosen uniformly from 0 to $\approx W_{v_t}$ and so the probability it lies in $[\l_{j-1},\l_{j-1}+\e_j]$ is at most $\approx \b_1/W_{v_s}$, which explains the term $m^2T_1^2\s_1$. If $s>t$ then we repeat the above argument with $[\l_{j-1},\l_{j-1}+\e_1]$ replaced by $[\l_{j}-\e_1,\l_{j}]$

The term $m^2(T_0-T_1)T_1\s_1$ arises in the same way with $j-1\in J_1, j\notin J_1$ or vice-versa.

The term $m^2(T_0-T_1)^2\s_2$ arises from the case where $j-1,j\notin J_1$. Here we can only assume that $W_j\gtrsim \log^{1/200}n/n^{1/2}$. This follows from \ref{P.wsa}, \ref{P.wsc} and Lemma \ref{lemd} and the fact that we exit the main loop with SUCCESS when we see a vertex of degree at least $n^{1/2}/\log^{1/100}n$. Assuming that $s<t$ we see that the probability that $\l_j$ lies in $[\l_{j-1},\l_{j-1}+\e_2]$ is at most $\b_2/W_{v_t}\sim \b_2/(\log^{1/200}n/n^{1/2})=o(1)$. 
\end{proof}

Given the Claim and \ref{P.wsc}, \ref{P.P} we have that w.h.p.
\beq{Fj}
W_{\vv{t}{j}-1}-W_{\vv{t}{j-1}+1}\geq \begin{cases}\b_1-\frac{\log n}{n^{1/2}}\geq \frac12\b_1&j\in J_1.\\\b_2-\frac{40m\log\log n}{n^{1/2}}\geq \frac12\b_2&j\notin J_1\end{cases}
\eeq
Now,
\beq{length}
W_{\vv{t}{j}-1}-W_{\vv{t}{j-1}+1}=\bfrac{\Upsilon_{m(\t(j)-1)}}{\Upsilon_{mn+1}}^{1/2}-\bfrac{\Upsilon_{m(\t(j-1)+1)}}{\Upsilon_{mn+1}}^{1/2}=\frac{\Upsilon_{m(\t(j)-1)}-\Upsilon_{m(\t(j-1)+1)}} {\Upsilon_{mn+1}^{1/2}(\Upsilon_{m(\t(j)-1)}^{1/2}+\Upsilon_{m(\t(j-1)+1)}^{1/2})}.
\eeq
Or,
\beq{lengthagain}
\sum_{u\in \WW{t}{j}}\eta_u=(W_{\vv{t}{j}-1}-W_{\vv{t}{j-1}+1})
\Upsilon_{mn+1}^{1/2}(\Upsilon_{m(\t(j)-1)}^{1/2}+\Upsilon_{m(\t(j-1)+1)}^{1/2})\geq \begin{cases}\b_1n^{1/2}&j\in J_1.\\\b_2n^{1/2}&j\notin J_1\end{cases}.
\eeq
It follows that w.h.p.
\beq{foy}
|\ww{t}{j-1}-\ww{t}{j}|=|\WW{t}{j}|\geq \begin{cases}\frac{\b_1n^{1/2}}{\log n}&j\in J_1.\\\frac{\b_2n^{1/2}}{40m\log\log n}&j\notin J_1\end{cases}.
\eeq
{\bf Case 3: $\ww{t}{j}\in V_L^{(t)},\ww{t}{j-1}\in V_R^{(t)}$}\\
Let $\th_v=\b_1,v\geq \log^{30}n$ and $\th_v=\b_2$ otherwise. We write
\begin{align}
\Pr(\exists s<t,v,k:v\in N_R(v_s),\l(v_t,k)\in I_{v}\pm \th_v\mid\cE,H_t)&\leq \sum_{s,t,v,k}\frac{\eta_{v_s}}{(vv_s)^{1/2}}\cdot\frac{w_{v}+2\th_v}{W_{v_t}}\\
&\lesssim \sum_{s,t,v,k}\brac{\frac{\eta_{v_s}\eta_v}{v(v_sv_t)^{1/2}}+ \frac{2n^{1/2}\th_v\eta_{v_s}}{v(v_sv_t)^{1/2}}}.\label{thth}
\end{align}
We bound the sum in the RHS of \eqref{thth} as follows: If $\max\set{v,v_s}\geq \log^{30}n$ then we bound the first sum by
$$(\#s,t,k)\cdot\log^2n\cdot\sum_{v=1}^n\frac{1}{v}\cdot\frac{1}{\log^{15}n}\leb (\om^2\log^2n)\cdot\log^2n\cdot\log n\cdot \frac{1}{\log^{15}n}=o(1).$$ 
We bound the second sum by
$$(\#s,t,k)\cdot2\log^{15/2}n\cdot\log n\cdot\sum_{v=1}^n\frac{1}{v}\cdot\frac{1}{\log^{15}n}\leq (\om^2\log^2n)\cdot\log^{15/2}n\cdot\log n\cdot\log n\cdot\frac{1}{\log^{15}n}=o(1).$$
When $\max\set{v,v_s,v_t}< \log^{30}n$ we bound the first sum by
\begin{multline}\label{ssum1}
(\#s,t,k)\cdot(40\log\log n)^2\cdot\sum_{v=1}^{\log^{30}n}\frac{1}{v}\cdot\frac{1}{\log^{1/200}n}\leb\\ (\om\log\log n)^2\cdot(\log\log n)^2\cdot\log\log n\cdot \frac{1}{\log^{1/200}n}=o(1).
\end{multline}
We bound the second sum by
\begin{multline}\label{ssum2}
(\#s,t,k)\cdot\log^{1/300}n\cdot40\log\log n\cdot\sum_{v=1}^{\log^{30}n}\frac{1}{v}\cdot\frac{1}{\log^{1/200}n}\leb\\ (\om\log\log n)^2\cdot\log^{1/300}n\cdot\log\log n\cdot \frac{1}{\log^{1/200}n}=o(1).
\end{multline}
Finally, if $\max\set{v,v_s}< \log^{30}n\leq \log^{30}n$ then we have to replace $(\#s,t,k)$ in \eqref{ssum1}, \eqref{ssum2} by $O(\om^2\log n\log\log n)$. But this is compensated by a factor $1/v_t^{1/2}\leq 1/\log^{15}n$.

It follows that \eqref{Fj} holds w.h.p. and the proof continues as for Case 2.

\proofend

\subsection*{\ref{step:23eta}}  
We now assume $t\leq T_0$. We begin by showing that DCA only uncovers a small part of the distribution of the $\eta$'s. 

Let $\Xi_t=P_t\cup N(P_t)$ and
$$S_{t,j}=\sum_{v\in \Xi_t}w_v.$$
\begin{lemma}\label{newW}
W.h.p., $S_{t,j}=o(W_j)$ for $\log^{1/100}n\leq j$ and $1\leq t\leq T_0$.  
\end{lemma}
\proofstart 
Assume first that $j\geq \log^{30}n$. It follows from \ref{P.wsa}, \ref{P.wsb}, \ref{P.P} and Lemma \ref{NRsize} that w.h.p. 
\beq{STis}
S_{t,j}\leq T_0\times \frac{\max \eta_{v_s}}{n^{1/2}}\times (m+\max |N_R(v_s)|)\lesssim\frac{T_0\log n(m+\log^3n)}{2mn^{1/2}}=O\bfrac{\om\log^5n}{n^{1/2}}.
\eeq
\beq{STit}
W_j\geq (1-o(1))\bfrac{j}{n}^{1/2}=\Omega\bfrac{\log^{10}n}{n^{1/2}}.
\eeq
This completes this case. Now assume that $j\leq \log^{30}n$.  \ref{P.wsa}, \ref{P.wsb}, \ref{P.wsc} and Lemma \ref{NRsize} that w.h.p. 
$$S_{t,j}\lesssim 40m\log\log n\times\frac{\om\log_{4/3}\log n}{2mn^{1/2}}\times (m+(\log\log n)^{20}) \ll W_j\sim\bfrac{j}{n}^{1/2}$$
for $\log^{1/100}n\leq j\leq \log^{30}n$.
\proofend

\subsubsection*{Dealing with left neighbors}
The calculation of the ratio $\rho_t$ takes contributions from two cases: where ${\vtt}$ is a left-neighbor of ${\vt}$, and where ${\vtt}$ is a right-neighbor of ${\vt}$.  
\begin{lemma}\label{34}
$$\E(\r_t\mathbf{1}_{{\vtt}<{\vt}}\mid{\Ht})\leq \frac{2}{3}.$$
\end{lemma}
\proofstart
Let $D$ denote the $(m/2)$th largest degree of a vertex in $N_R({\vt})$. We write
\begin{align}
\E(\r_t\mathbf{1}_{{\vtt}<{\vt}}\mid{\Ht})&= \sum_{d}\E(\r_t\mathbf{1}_{{\vtt}<{\vt}}\mid{\Ht}\mid D=d)\Pr(D=d)\\
&\leq \sum_d\E\bfrac{\z_d}{\vt}\Pr(D=d)\\
&=\E\bfrac{\z_D}{\vt},
\end{align}
where $\z_d$ is the index of the smallest degree left neighbor of $\vt$ that has degree at least $d$. We let $\z_d=0$ if there are no such left neighbors. We now couple $\z$ with a random variable that is independent of the algorithm and can be used in its place.

Going back to Section \ref{s.BR} let us associate $\ell_k$ for $k\geq \om$ with an index $\m_k$ chosen uniformly from $[\rdown{k/m}]$. In this way, vertex $i\geq\om$ is associated with $m$ uniformly chosen vertices $a_{i,1},a_{i,2},\ldots,a_{i,m}$ in $[i-1]$. Furthermore, we can couple these choices so that if $N_L(i)=\set{b_{i,1},b_{i,2},\ldots,b_{i,m}}$ then given $\ff$ we have (i) $\Pr(b_{i,j}\leq a_{i,j})\geq 1-o(1)$ and (ii) $b_{i,j}\leq 2a_{i,j}$ for all $i,j$. This because $\Pr(b_{i,j}\leq k)\sim W_k/W_i\sim (k/i)^{1/2}$ (giving (i)) and $(k/i)^{1/2}\geq k/i$ and $(1-o(1))(k/i)^{1/2}\geq k/2i$ (giving (ii)).

So now let $\m$ be the index of the uniform choice associated with the largest degree left neighbor of $\vt$ that has degree at least $D$. Thus 
$$\E\bfrac{\z_D}{\vt}\lesssim\E\bfrac{\m}{\vt}=\frac{1}{2}+o(1)\leq \frac 2 3.$$
\proofend
\subsubsection*{Dealing with right neighbors} 
It will be more difficult to consider the contribution of right-neighbors.  In preparation, for $\l_0\leq \g\leq 1-1/m$ we define 
$$\D_\g^i:=m+\g m\z(i)$$
where $\z(i),\z^+(i)$ are defined in \eqref{zetax}, \eqref{zeta+x} respectively.
We note that $\eta_i\z(i)$ is a lower bound for the expected degree of vertex $i,i\geq \om$, see Lemma \ref{lemd}(a). Note also that $\eta_i\z^+(i)$ is an upper bound for the expected degree of vertex $i,i\geq \om$.

The parameter $\D^i_\g$ is a degree threshold.  For a suitable parameter $\g$, we wish it to be the case that there should be many left-neighbors but few right-neighbors which have degree greater than $\D_\g^i$.
We define
\beq{gammastar}
\g_i^*=\max\set{\g:\card{\set{j\in N_L(i):d_n(j)\geq\D_\g^i}}\geq m/2}.
\eeq
$\Delta^{{\vt}}_{\g_{{\vt}}^*}$ is a lower bound on the degree needed for vertex $j>{\vt}$ to be considered by {\sc DCA} as the next vertex; thus we proceed by analyzing the distribution of $\g_{{\vt}}^*$.  We first derive upper bounds for $\Pr(\g_{{\vt}}^*\leq\g\mid{H_t})$.

\begin{lemma}\label{Pg<}
There exists $c_1>0$ such that 
\begin{align}
\Pr(\g^*_{{\vt}}\leq \g\mid{H_t})&\lesssim (\g^{1/2}e^{1-\g^{1/2}})^{c_1m^2}+me^{-c_1\g^{1/2}m},\; 0\leq \g\leq \frac18.\hspace{1in}\label{P<g}\\
\Pr(\g^*_{{\vt}}\leq \g\mid{H_t})&\lesssim (\g^{1/2}e^{1-\g^{1/2}})^{c_1m^2}+me^{-c_1/\g^{1/2}},\; 0\leq \g\leq \frac18.\label{P<gg}\\
\Pr(\g^*_{{\vt}}\leq \tfrac 54\mid{H_t})&\lesssim e^{-c_1m}.\label{P<g0}\\
\Pr(\g^*_{{\vt}}\geq \g\mid{H_t})&\lesssim \g^{-c_1m},\quad \g\geq 10^5.\label{P>g}
\end{align}
\end{lemma}
\proofstart  For $j<v_t$, we define events $\cA_j=\set{\eta_j\leq \g^{1/2} m}$ and $\cD_j=\set{d_n(j)\leq\D_\g^{{\vt}}}$. We need to estimate $\Pr\brac{\bigcap_{j\in S}\cD_j}$ for subsets $S\subseteq N_L({\vt})$ of size $m/2$. We write
\beq{venn}
\bigcap_{j\in S}\cD_j\subseteq \bigcap_{j\in S}(\cA_j\cup (\bar{\cA}_j\cap \cD_j))\subseteq \bigcap_{j\in S}\cA_j\cup\bigcup_{j\in S}(\bar{\cA}_j\cap \cD_j).
\eeq
Now, using inequality \eqref{l.etamx} and equation \eqref{uneta}, we see that if $0\leq \g\leq 1/8$ then for $j<{\vt}$,
\beq{venn1}
\Pr(\eta_j\leq \g^{1/2}m\mid{H_t})\lesssim m(\g^{1/2}e^{1-\g^{1/2}})^m.
\eeq
The RHS of \eqref{venn1} includes a factor of $1+o(1)$ due to conditioning on $\cE,H_t$. 

So,
\beq{v1}
\Pr\brac{\bigcap_{j\in S}\cA_j\ \bigg|\ {H_t}}\lesssim (m(\g^{1/2}e^{1-\g^{1/2}})^m)^{|S|}.
\eeq
Furthermore, because $j\in N_L({\vt})$ implies that $i\geq j$ and hence $\z({\vt})\leq \z(j)$,
\begin{align}
\Pr((d_n(j)\leq\D_\g^i)\wedge\bar{\cA}_j\mid{H_t})&\leq \Pr\brac{d_n(j)-m\leq \g m\z(j)\mid {H_t}}\Pr(\eta_j>\g^{1/2}m|{H_t})\\
&\lesssim\exp\set{-\frac{(1-\g^{1/2})^2}{2}\g^{1/2}m\z(j)}.\label{venn2}\\
\end{align}
{\bf Explanation of \eqref{venn2}:} We remark first that the conditioning on ${\cE,H_t}$ only adds a $(1+o(1))$ factor to the upper bound on our probability estimate. We now apply Lemma \ref{lemd}(b) with $1-\a=\g$ and $\eta_j\geq \g^{1/2}m$.

From \eqref{venn} (summing over all $m/2$ subsets of $N_L({\vt})$) and \eqref{venn2} (summing over $N_L({\vt})$) we obtain 
\begin{multline}\label{vv2}
\Pr(\g_{{\vt}}^*\leq\g\mid{H_t})=\Pr(|\set{j\in N_L({\vt}):d_n(j)\leq\D_\g^i}|\geq m/2|\ \mid{H_t})\lesssim\\
2^m\brac{(m(\g^{1/2} e^{1-\g^{1/2}})^m)}^{m/2}+ m\exp\set{-\frac{(1-\g^{1/2})^2}{2}\g^{1/2}m\z({\vt})}.
\end{multline}
We observe that $j\in N_L({\vt})$ implies that $d_n(j)\geq m+1$. So,
\beq{gasp}
m+1<\D_\g^i\text{ implies } \z({\vt})> \frac{1}{m\g}.
\eeq
Using \eqref{gasp} in \eqref{vv2} verifies \eqref{P<gg}, after bounding $2^m\brac{(m(\g^{1/2} e^{1-\g^{1/2}})^m)}^{m/2}$ by $(\g^{1/2}e^{1-\g^{1/2}})^{c_1m^2}$. 

From \eqref{venn} and \eqref{venn2},
\begin{multline}\label{v2}
\Pr(\g_{{\vt}}^*\leq\g\mid{H_t})\lesssim (m(\g^{1/2} e^{1-\g^{1/2}})^m)^{m/2}+\\
+\Pr(\cA_1\mid{H_t})+\Pr(\bar{\cA}_1\mid{H_t})^{-1}m\exp\set{-\frac{(1-\g^{1/2})^2}{2}\g^{1/2}m\z\bfrac{9n}{10}}
\end{multline}
Here 
$$\cA_1=\set{\left|N_L({\vt})\cap \left[\frac{9n}{10}\right]\right|\geq \frac{m}{4}}.$$
{\bf Explanation of \eqref{v2}:} The first term is from \eqref{v1}. If $\bar{\cA}_1$ holds then $v_t$ has at least one left neighbor $j\leq 9n/10$. The final term comes from using \eqref{venn2} and $\z(j)\geq \z(9n/10)$. The factor $\Pr(\bar{\cA}_1)^{-1}$ handles the conditioning on $\bar{\cA}_1$. The factor $m$ is the union bound for choices of $j$.

Now $\left|N_L({\vt})\cap \left[\frac{9n}{10}\right]\right|$ is dominated by the binomial $Bin(m,n/10)$ and so $\Pr(\cA_1\mid H_t)\leq e^{-d_1m}$. Now $\z(9n/10)\geq 1/20$ and plugging these facts into \eqref{v2} yields \eqref{P<g}. Here we have absorbed the $e^{-d_1m}$ term into $me^{-c_1\g^{1/2}m}$ and we will do so again below.

We continue with the proof of \eqref{P<g0}. For $j\in N_L(\vt)$, we observe that if $d_n(j)\leq \D_\g^\vt$ and $\g\leq \frac 5 4$ then 
\[
d_n(j)-m\leq \frac{5m}4\z({\vt}).
\]
We now estimate the probability that a uniform random choice of $j\in N_L(\vt)$ (for fixed $H_t$, which determines $\vt$) has certain properties.

We first observe that 
\beq{eq25}
\Pr\brac{j\geq \frac{3i}{5}\bigg| H_t}\lesssim \brac{1-\frac{W_{3i/5}}{W_{{\vt}}}}\sim\brac{1-\bfrac{3}{5}^{1/2}}<\frac25.
\eeq

(For this we used \ref{P.wsa}.)

Now \eqref{l.etamx} implies that 
\beq{eqd2}
\Pr(\eta_j\leq 0.99m\mid H_t)\leq e^{-d_2m}.
\eeq
Moreover, for $\eta_j>0.99m$ and $j<3\vt/5$, we have 
\[
\frac{\z(j)}{\z(v_t)}=\bfrac{\vt}{j}^{1/2} \bfrac{1-\bfrac{j}{n}^{1/2}-\e}{1-\bfrac{\vt}{n}^{1/2}-\e}\geq \bfrac{\vt}{j}^{1/2}
\]
where 
\beq{eps}
\e=\frac{5L\log\log n}{\om^{3/4}\log n}.
\eeq

Thus we have
$$\E(d_n(j)-m\mid H_t)\geq \eta_j\z(j)\gtrsim  0.99m\bfrac{5}{3}^{1/2}\z({\vt}).$$

Now $0.99\times (5/3)^{1/2}=1.278..>1.01\times 5/4$ and so 
\beq{hell}
\Pr\brac{d_n(j)-m\leq \frac{5\z({\vt})}{4}\ \bigg|\ {H_t}}\leq \Pr\brac{d_n(j)-m\leq \frac{\z(j)\eta_j}{1.01}\ \bigg|\ {H_t}}\leq e^{-d_3 \eta_j \zeta(j)}\leq e^{-d_4 m}
\eeq
using Lemma \ref{lemd}(b). It follows from \eqref{eq25} and \eqref{eqd2} and \eqref{hell} that
$$\Pr\brac{\g_{{\vt}}^*<\frac54\ \bigg|\ {H_t}}\leq \Pr\brac{Bin\brac{m,e^{-d_2m}+e^{-d_4m}+\frac25}\geq \frac{m}2}\leq e^{-d_3m}.$$
This completes the proof of \eqref{P<g0}.

To deal with \eqref{P>g} we observe that if $d_n(j)\geq \D_\g^\vt$ and $\g\geq10^5$ then
$$j\in N_L({\vt})\text{ and }j\leq \frac{\vt}{\g^{1/2}}\quad\text{ or }\quad\eta_j\geq \g^{1/2} m\bfrac{j}{\vt}^{1/2}\geq \g^{1/4}m\quad\text{or}\quad d_n(j)-m\geq \g^{3/4}\eta_j\z(j).$$
But 
\beq{jbad1}
\Pr\brac{j\in N_L({\vt})\text{ and }j\leq \frac{\vt}{\g^{1/2}}\ \bigg|\ {H_t}}\lesssim \frac{W_{\vt/\g^{1/2}}}{W_{{\vt}}}\lesssim \frac{1}{\g^{1/4}}.
\eeq
And, using \ref{P.wsb} and $\g\geq 10^5$,
\begin{align}
\Pr\brac{\eta_j\geq \g^{1/4}m\ \bigg|\ {H_t}}&\lesssim \sum_{l=2\vt/\g}^\vt\frac{w_l}{W_{{\vt}}}\int_{\eta_l=\g^{1/4}m}^\infty \frac{\eta_l^me^{-\eta_l}}{(m-1)!}d\eta_l\\
&\lesssim \sum_{l=2\vt/\g}^\vt\frac{e^{-\g^{1/4}m}}{2(\vt l)^{1/2}}\\
&\lesssim e^{-\g^{1/4}m}.\label{48}
\end{align}
Lastly, using \eqref{eq25}, \eqref{eqd2} and Lemma \ref{lemd}(d) and $\z^+(j)\lesssim \z(j)$ for $j\leq 3n/5$ we have
\beq{49}
\Pr\brac{d_n(j)-m\geq \g^{3/4}\eta_j\z(j)\ \mid\ {H_t}}\leq \frac25+e^{-d_2m}+e^{-\g^{3/4}\eta}\leq 0.41.
\eeq
It follows from \eqref{jbad1}, \eqref{48} and \eqref{49} that
$$\Pr\brac{\g_{{\vt}}^*\geq\g\mid{H_t}}\lesssim \Pr\brac{Bin\brac{m,(1+o(1))\brac{\frac{1}{\g^{1/4}}+e^{-\g^{1/4}m}+0.41}}\geq \frac{m}{2}}\lesssim e^{-d_3m}.$$
This completes the proof of the lemma.
\proofend

\begin{corollary}\label{cor1}
W.h.p. $\g_{v_s}^*\geq 1/(\log\log n)^2$ for $s=1,2,\ldots,T=O(\log n)$. 
\end{corollary}
\proofstart
The value of $\g_{v_s}^*$ is determined when $v_s$ is first visited and in this case we can apply Lemma \ref{Pg<}. In which case the result follows directly from \eqref{P<gg}.
\proofend

We now have a handle on the distribution of $\g_{{\vt}}^*$. We now put bounds on the expected number of $j>\vt$ that can be considered to be a candidate for ${\vtt}$, conditioned on the value of $\gamma_{\vt}^*$.  In particular, we let
$$D_\g^{\vt}=\set{j>\vt:d_n(j)\geq \D_\g^{\vt}}.$$
We will bound the size of $D_\gamma^i$ by dividing $D_\g^i$ into many parts bounding each part; in particular, $\k\in \mathbb{N}$ we let 
\beq{defJ}
J_\g^{i,\k}=\begin{cases}
\left[i,\frac{i}{\g^2}\right]\cap D_\g^i&\k=0.\\ 
\left[\frac{i}{\g^2}\brac{1+\frac{\k-1}{L}},\frac{i}{\g^2}\brac{1+\frac{\k}{L}}\right]\cap D_\g^i&1\leq \k\leq \frac{2n\g^2L}{i}.
\end{cases}
\eeq 
Note that $J_\g^{i,0}=\emptyset$ if $\g\geq1$.

Finally, we let
\beq{rs}
r_\g^{i,\k}:=|J_\g^{i,\k}|\quad \text{and}\quad r_\g^i:=\sum_{\k\geq 0}r_\g^{i,\k}\quad\quad \mbox{and}\quad s^i_\g:=\sum_{\k\geq 0}\sum_{j\in J_\g^{i,\k}}j\leq \frac{i}{\g^2}\sum_{\k\geq 0}\brac{1+\frac{\k+1}{L}}r_\g^{i,\k}.
\eeq
\begin{remark}\label{remboo}{\it
We have that $\E\brac{\frac{2}{m} \frac{1}{v_t} s_\g^{v_t}\mid H_t}$ is an upper bound on the expectation of the ratio $\rho_t=\tfrac{v_{t+1}}{v_t}$, conditioned on the event that $v_{t+1}>v_t$, since each right neighbor whose index is included in the sum $s_\g^{v_t}$ has probability of at most $\frac 2 {m}$ of being chosen by the algorithm.
}
\end{remark}

\begin{lemma}\label{l.right}\ 
If ${\vt}\leq n\brac{1-\frac{3L^3}{m}}$ then
\beq{RK5b}
\E(r_\g^{{\vt}}\mid {\Ht})\leq \frac{\eta_{{\vt}}}{\g L}\brac{L\max\set{0,(1-\g)}+7+10Le^{-c_2\g L}}.
\eeq
\beq{RK55b}
\E\brac{\frac{s_\g^{{\vt}}}{\vt}\mid {\Ht}}
\leq \frac{\eta_{{\vt}}}{\g^3L}\brac{L\max\set{0,(1-\g)}+13+100Le^{-c_2\g L}}.
\eeq
Moreover, 
\beq{RKxx}
\text{If }\vt\leq n/5\text{ and }\k\geq (\log\log n)^4\text{ and }\g\geq 1/(\log\log n)^2\text{ then }\Pr(r_\g^{{\vt},\k}>0)\leq\frac{1}{\log^2n}. 
\eeq
\end{lemma}
Note that \eqref{RKxx} implies the second inequality in \eqref{ratio}.

\proofstart
Recall from Lemma \ref{newW} that w.h.p., 
\beq{1234}
S_{T,j}=o(W_j)\text{ for }j\geq \log^{1/100}n.
\eeq
We write 
\begin{align}
&\E(r_\g^{{\vt},\k}\mid{\Ht})\nonumber\\
&\lesssim \sum_{j\in J_\g^{{\vt},\k}}  \frac{mw_{{\vt}}}{W_j-S_{t,j}}\int_{\eta_j=0}^\infty
  \frac{\eta_j^{m-1}e^{-\eta_j}}{(m-1)!}\Pr\brac{d_n(j)\geq \D_\g^{{\vt}}\mid{\Ht},\eta_j}d\eta_j\label{dood0}\\
&\lesssim \eta_{{\vt}}\sum_{j\in J_\g^{{\vt},\k}}  \frac{1}{2({\vt}j)^{1/2}}\int_{\eta_j=0}^\infty
  \frac{\eta_j^{m-1}e^{-\eta_j}}{(m-1)!}\Pr\brac{d_n(j)\geq \D_\g^{{\vt}}\mid{\Ht},\eta_j}d\eta_j.\label{dood1}
\end{align}
{\bf Explanation of \eqref{dood0} and \eqref{dood1}:}
We sum over the relevant $j$ and fix $\eta_j$. We multiply by the density of $\eta_j$ and integrate. Using \eqref{1234} we see that 
$$\frac{mw_{{\vt}}}{W_j-S_{t,j}}\sim\frac{mw_{{\vt}}}{W_j}\sim\frac{\eta_{{\vt}}}{2({\vt}j)^{1/2}}.$$ 
This is asymptotically equal to the expected number of times $j$ chooses $v_t$ as a neighbor.

Thus
\beq{cond00}
\E(r_\g^{{\vt},\k}\mid{\Ht})\lesssim \sum_{j\in J_\g^{{\vt},\k}}\frac{\eta_{{\vt}}}{2({\vt}j)^{1/2}}I_j
\eeq
where
$$I_j=\int_{\eta_j=0}^\infty \frac{\eta_j^{m-1}e^{-\eta_j}}{(m-1)!}\Pr\brac{d_n(j)\geq \D_\g^{{\vt}}\mid {\Ht},\eta_j}d\eta_j\leq 1.$$
If $m$ is large then
\beq{RK0}
\E\brac{r_\g^{{\vt},0}\mid{\Ht}}\lesssim\begin{cases}0&\g\geq 1.\\\sum_{j\in J_0^{{\vt}}(\g)}\frac{\eta_{{\vt}}}{2({\vt}j)^{1/2}}\lesssim \eta_{{\vt}}\frac{1-\g}{\g}&\g<1.\end{cases}
\eeq 
Continuing, for $\k\geq 1$, we write:
\beq{A123}
I_j\lesssim A_1+A_2+A_3,
\eeq
where
\begin{align*}
A_1&=\int_{\eta_j=0}^{(1-1/L)\g m}\frac{\eta_j^{m-1}e^{-\eta_j}}{(m-1)!} \Pr\brac{d_n(j)\geq \D_\g^{{\vt}}\mid\eta_j}d\eta_j.\\
A_2&=\int_{\eta_j=(1-1/L)m}^{(1+1/L)m}\frac{\eta_j^{m-1}e^{-\eta_j}}{(m-1)!}\Pr\brac{d_n(j)\geq \D_\g^{{\vt}}\mid\eta_j}d\eta_j\\
A_3&=\int_{\eta_j=(1+1/L)m}^\infty\frac{\eta_j^{m-1}e^{-\eta_j}}{(m-1)!}\Pr\brac{d_n(j)\geq \D_\g^{{\vt}}\mid\eta_j} d\eta_j.
\end{align*}
and then write $r_\g^{{\vt},\k}=r_\g^{{\vt},\k,1}+r_\g^{{\vt},\k,2}+r_\g^{{\vt},\k,3}$. Here $r_\g^{{\vt},\k,l}$ is equal to the RHS of \eqref{cond00} with $I_j$ replaced by $A_l$. The implicit $(1+o(1))$ factor in \eqref{A123} arises from replacing\\ $\Pr\brac{d_n(j)\geq \D_\g^{{\vt}}\mid\eta_j,{\Ht}}$ by $\Pr\brac{d_n(j)\geq \D_\g^{{\vt}}\mid\eta_j}$ in the integrals, i.e., ignoring the conditioning due to ${\Ht}$. Since $j>{\vt}$, the only effect of ${\Ht}$ is on $W_j$ through $w_{{\vt}}$. Here we have that w.h.p. $W_j\sim \bfrac{{\vt}}{n}^{1/2}$ and $w_{{\vt}}\sim\frac{\eta_{v_t}}{2m({\vt}n)^{1/2}}=O\bfrac{\log n}{2m({\vt}n)^{1/2}}=o(W_{{\vt}})$.

{\bf Case 1: $n_1\leq {\vt}<n/5$}:\\
Note that in this case
\beq{zeta}
\z({\vt})\geq \frac12\bfrac{n}{{\vt}}^{1/2}\ge1.
\eeq
In the following we use Lemma \ref{lemeta} to estimate the integrals over $\eta_j$. 
We observe that
\begin{align}
&\E(r_\g^{{\vt},\k,1}\mid{\Ht})\\
&\lesssim \sum_{j\in J_\g^{{\vt},\k}}\frac{\eta_{{\vt}}}{2({\vt}j)^{1/2}}
\int_{\eta_j=0}^{\brac{1-\frac1L}m}\frac{\eta_j^{m-1}e^{-\eta_j}}{(m-1)!}\Pr\brac{d_n(j)\geq \D_\g^{{\vt}}\mid H_t,\eta_j}d\eta_j,\\
&\lesssim \sum_{j\in J_\g^{{\vt},\k}}\frac{\eta_{{\vt}}}{2({\vt}j)^{1/2}} \int_{\eta_j=0}^{\brac{1-\frac1L}m}\frac{\eta_j^{m-1}e^{-\eta_j}}{(m-1)!} 
\cdot \left(\begin{cases}1&1\leq\k\leq 10L,\\ (e(L/\k)^{1/2})^{d_0\g m\z({\vt})} &\k>10L \end{cases}\right)d\eta_j,\label{HOEF}
\end{align}
{\bf Explanation of \eqref{HOEF}:} We remark first that the conditioning on ${H_t}$ only adds a $(1+o(1))$ factor to the upper bound on our probability estimate. We will use Lemma \ref{lemd} to bound the probability that degrees are large. Now with our bound on $v_t$ and within the range of integration, the ratio of $\D_\g^{{\vt}}-m$ to the mean of $d_n(j)-m$ is
\begin{multline}\label{asdf1}
\frac{\D_\g^{{\vt}}-m}{\z^+(j)}= \frac{\g m\bfrac{n}{{\vt}}^{1/2} \brac{1-\bfrac{{\vt}}{n}^{1/2}-o(1)}}{\eta_j\bfrac{n}{j}^{1/2}\brac{1-\bfrac{j}{n}^{1/2}+o(1)}}\gtrsim \frac{\g m}{\eta_j}\bfrac{j}{{\vt}}^{1/2}\geq \\ \frac{L}{L-1}\brac{1+\frac{\k-1}{L}}^{1/2}\geq \bfrac{\k}{L}^{1/2}\text{ when }\k\geq10L.
\end{multline}
We then use \eqref{largei} and Lemma \ref{lemd}(d) with $\b=(\k/L)^{1/2}$.

Continuing, we observe that
\beq{width}
\brac{1+\frac{\k}{L}}^{1/2}-\brac{1+\frac{\k-1}{L}}^{1/2}\leq \frac{1}{2L}
\eeq
and so
\begin{align}
\E(r_\g^{{\vt},\k,1}\mid{\Ht})&\leq \frac{\eta_{v_t}e^{-m/(2L^2)}}{\g {\vt}^{1/2}} \brac{\brac{\brac{1+\frac{\k+1}{L}}{\vt}}^{1/2}-\brac{\brac{1+\frac{\k-1}{L}}{\vt}}^{1/2}}\times\\ &\hspace{2in}\left(\begin{cases}1&1\leq\k\leq 10L,\\ (e(L/\k)^{1/2})^{d_0\g m\z({\vt})} &\k>10L, \end{cases} \right)\\
&\leq \frac{\eta_{v_t}e^{-m/(2L^2)}}{\g L}\cdot \left(\begin{cases}1&1\leq \k\leq 10L,\\ (e(L/\k)^{1/2})^{d_0\g m\z({\vt})} &\k>10L. \end{cases}\right)\label{all1}
\end{align}

Continuing, it follows from \eqref{width} that  
\beq{widtha}
\sum_{j\in J_\g^{{\vt},\k}}\frac{1}{j^{1/2}}\leq \frac{{\vt}^{1/2}}{\g L}.
\eeq
\begin{align}
&\E(r_\g^{{\vt},\k,2}\mid{\Ht})\\
&\leq \sum_{j\in J_\g^{{\vt},\k}}\frac{\eta_{{\vt}}}{2({\vt}j)^{1/2}}
\int_{\eta_j=(1-1/L)m}^{(1+1/L)m}\frac{\eta_j^{m-1}e^{-\eta_j}}{(m-1)!}\Pr\brac{d_n(j)\geq \D_\g^{{\vt}}\mid{\Ht},\eta_j} d\eta_j\\
&\leq \sum_{j\in J_\g^{{\vt},\k}}\frac{\eta_{{\vt}}}{2({\vt}j)^{1/2}}
\times\begin{cases}1&\k\leq 3,\\\exp\set{-\frac{d_1\g m\z({\vt})}{L^2}}&4\leq \k\leq 10L,\\ (e(L/\k)^{1/2})^{d_0\g m\z({\vt})} &\k>10L, \end{cases}\label{112}\\
&\leq \frac{\eta_{{\vt}}}{\g L}\times\begin{cases}1&1\leq\k\leq 3,\\e^{-d_1\g m\z({\vt})/L^2}&4\leq \k\leq 10L,\\(e(L/\k)^{1/2})^{d_0\g m\z({\vt})} &\k>10L. \end{cases}\label{gggg}\\
\end{align}
where we have used \eqref{widtha}. 

{\bf Explanation for \eqref{112}:} We proceed in a similar manner to \eqref{asdf1} and use
$$\frac{\D^{v_t}_\g-m}{\z^+(j)}=\frac{\g m\bfrac{n}{{\vt}}^{1/2}\brac{1-\bfrac{{\vt}}{n}^{1/2}-o(1)}}{\eta_j\bfrac{n}{j}^{1/2} \brac{1-\bfrac{j}{n}^{1/2}+\e}}\geq 1+\frac{1}L\text{ if }\k\geq 4,\eta_j\geq \brac{1-\frac1L}m.$$
Then we use Lemma \ref{lemd}(c),(d).

Continuing,
\beq{ffff}
\E(r_\g^{{\vt},\k,3}\mid{\Ht})
\leq \sum_{j\in J_\g^{{\vt},\k}}\frac{\eta_{{\vt}}}{2({\vt}j)^{1/2}}
\int_{\eta_j=(1+1/L)m}^{\infty}\frac{\eta_j^{m-1}e^{-\eta_j}}{(m-1)!} \Pr\brac{d_n(j)\geq \D_\g^{{\vt}}\mid{\Ht},\eta_j} d\eta_j
\eeq
We bound the integral in \eqref{ffff} by something independent of $j$ and then as above, there is a factor $\eta_{v_t}/\g L$ arising from the sum over $j$.

For all $1\leq \k\leq 80L+1$, we simply use the bound
\beq{1122}
\int_{\eta_j=m\brac{1+\frac{1}{L}}}^\infty\frac{\eta_j^{m-1}e^{-\eta_j}}{(m-1)!}d\eta_j\leq \exp\set{-\frac{d_2m}{L^2}}.
\eeq
For $\k\geq 80L+2$, we split the integral from \eqref{ffff} into pieces $B_1^\k,B_2^\k$ (whose definition depends on $\k$), which we will bound individually.  

In particular, we use
\begin{align}
B_1^{\k}&=\int_{\eta_j=m\brac{1+\frac{\k-1}{L}}^{1/4}}^{\infty}\frac{\eta_j^{m-1}e^{-\eta_j}}{(m-1)!}\Pr\brac{d_n(j)\geq \D_\g^{{\vt}}\mid{\Ht},\eta_j}d\eta_j\\
&\leq  \int_{\eta_j=m\brac{1+\frac{\k-1}{L}}^{1/4}}^{\infty}\frac{\eta_j^{m-1}e^{-\eta_j}}{(m-1)!}d\eta_j\\
&\leq e^{-d_3m(\k/L)^{1/4}}\label{42a}
\end{align}
and 
\begin{multline}
B_2^\k=\int_{\eta_j=m\brac{1+\frac{1}{L}}}^{m\brac{1+\frac{\k-1}{L}}^{1/4}}\frac{\eta_j^{m-1}e^{-\eta_j}}{(m-1)!} \Pr\brac{d_n(j)\geq \D_\g^{{\vt}}\mid{\Ht},\eta_j}d\eta_j\\
\leq  \Pr\brac{d_n(j)\geq \D_\g^{{\vt}}\bigg|\ {\Ht},\eta_j\leq m\brac{1+\frac{\k-1}{L}}^{1/4}}\leq 
\bfrac{eL^{1/4}}{\k^{1/4}}^{d_4\g m\z({\vt})}\label{43}
\end{multline}
to bound the integral in \eqref{ffff} by $B_1^\k+B_2^\k$ for all $\k\geq 80L+2$.
 
Therefore, gathering the many terms together (and using that $\k\leq \frac{2nL\g^2}{{\vt}}$ from \eqref{defJ}) and relying on $m$ large to allow crude upper bounding, we see that
\begin{multline}
\frac{\g L}{\eta_{v_t}}\E(r_\g^{{\vt}}\mid{\Ht})\lesssim L\max\set{0,(1-\g)}[\text{from \eqref{RK0}}]+10Le^{-m/(2L^2)}[\text{from \eqref{all1}}]\\
 +10Le^{-d_1\g m\z({\vt})/L^2}[\text{from \eqref{gggg}}]+
\brac{2+e^{-m/(2L^2)}} \sum_{\k=10L}^{2nL\g^2/i}\bfrac{eL^{1/2}}{\k^{1/2}}^{d_0\g m\z({\vt})}[\text{from \eqref{all1} and \eqref{gggg}}]\\
+6[\text{from \eqref{gggg}}]+100L\exp\set{-\frac{d_2m}{L^2}}[\text{from \eqref{1122}}]\\ +\sum_{\k=80L+2}^{2nL\g^2/i}\brac{e^{-d_3m(\k/L)^{1/4}}+\bfrac{eL^{1/4}}{\k^{1/4}}^{d_4\g m\z({\vt})}}[\text{from \eqref{42a} and \eqref{43}}].\label{14-}
\end{multline}
We first observe that if $\frac{n}{{\vt}}<\frac{10}{\g^2}$ then the summations $\k=10L,\ldots, 2nL\g^2/{\vt}$ etc. above are empty. For larger $n/{\vt}$ we can therefore assume that $\g m(n/{\vt})^{1/2}\geq m$ which implies (see \eqref{zeta}) that $\g m\z({\vt})\geq m/2$ and then we can assume that 
\beq{pq1}
\sum_{\k=10L}^{2nL\g^2/{\vt}}\bfrac{eL^{1/2}}{\k^{1/2}}^{d_0\g m\z({\vt})}\leq \frac{1}{1000}\text{ and }\sum_{\k=80L+1}^{2nL\g^2/{\vt}}\brac{\frac{eL^{1/4}}{\k^{1/4}}}^{d_4\g m\z({\vt})}\leq \frac{1}{1000}.
\eeq
Plugging these estimates into \eqref{14-} and making some simplifications, we obtain \eqref{RK5b}.

Going back to \eqref{rs} we have
\begin{multline}
\frac{\g^3 L}{\eta_{{\vt}}}\E\brac{\frac{2s_\g^{{\vt}}}{mi}\ \bigg|\ {\Ht}}\leq L\max\set{0,(1-\g)}\\ +
200Le^{-m/(2L^2)}+100Le^{-d_1\g m\z({\vt})/L^2}+\brac{2+e^{-m/(2L^2)}} \sum_{\k=10L}^{2nL\g^2/v_t}\frac{2\k}{L}\bfrac{eL^{1/2}}{\k^{1/2}}^{d_0\g m\z({\vt})}\\ +12+10^4L\exp\set{-\frac{d_2m}{L^2}}
+\sum_{\k=80L+2}^{2nL\g^2/v_t}\frac{2\k}{L}\brac{e^{-d_3m(\k/L)^{1/4}}+\brac{\frac{eL^{1/4}}{\k^{1/4}}}^{d_4\g m\z({\vt})}}.\label{13}
\end{multline}
Making similar estimates to what we did for \eqref{pq1} gives us \eqref{RK55b}. 

We obtain \eqref{RKxx} from \ref{P.P}, \eqref{all1}, \eqref{gggg}, \eqref{42a} and \eqref{43}. Indeed, if $J_\g^{v_t,\k}\neq\emptyset$ then from its definition  we must have $v_t\leq \frac{2L\g^2n}{\k-1}$. Together with $v_t\leq n/5$ we obtain that $\z(v_t)\geq \frac{\k^{1/2}}{2L^{1/2}\g}$. Thus, in this case,
\beq{jkjk}
\bfrac{eL^{1/2}}{\k^{1/2}}^{d_0\g m\z({\vt})}\leq \bfrac{eL^{1/2}}{(\log\log n)^2}^{d_0m(\log\log n)^2/2L^{1/2}}=o\bfrac{1}{\log^{10}n}.
\eeq
This deals with the probabilities in \eqref{all1} and \eqref{gggg}. For \eqref{gggg} we rely $m$ large to to show that $e^{-d_3m(\k/L)^{1/4}}=o(1/\log^{10}n)$. Equation \eqref{42a} is dealt with in a similar manner to \eqref{all1}. Here we have $\bfrac{eL^{1/4}}{\k^{1/4}}^{d_4\g m\z({\vt})}$ which is the square root of \eqref{jkjk}.

{\bf Case 2: $n/5\leq {\vt}\leq n\brac{1-\frac{3L^3}{m}}$}:\\
The upper bound on $v_t$ implies that
$$m\z({\vt})\geq L^3.$$
Using the same definitions of $r_\g^{{\vt},\k,l},l=1,2,3$ as above:
\begin{align}
\sum_{\k\geq1}\E(r_\g^{{\vt},\k,1}\mid{\Ht})&\leq\sum_{\k\geq1} \sum_{j\in J_\g^{{\vt},\k}}\frac{\eta_{{\vt}}}{2({\vt}j)^{1/2}}
\int_{\eta_j=0}^{\brac{1-\frac1L}m}\frac{\eta_j^{m-1}e^{-\eta_j}}{(m-1)!}d\eta_j,\\
&\leq\sum_{\k\geq1} \sum_{j\in J_\g^{{\vt},\k}}\frac{\eta_{{\vt}}}{2({\vt}j)^{1/2}}e^{-m/(2L^2)},\qquad\text{ from Lemma \ref{lemeta}(e)},\\
&\leq \frac{\eta_{{\vt}}}{\g}\bfrac{n}{{\vt}}^{1/2}e^{-m/(2L^2)}\\
&\leq \frac{5^{1/2}\eta_{{\vt}}}{\g}e^{-m/(2L^2)}.
\end{align}
\begin{align}
&\sum_{\k\geq1}\E(r_\g^{{\vt},\k,2}\mid{\Ht})\\
&\lesssim \sum_{\k\geq1}\sum_{j\in J_\g^{{\vt},\k}}\frac{\eta_{{\vt}}}{2({\vt}j)^{1/2}}
\int_{\eta_j=(1-1/L)m}^{(1+1/L)m}\frac{\eta_j^{m-1}e^{-\eta_j}}{(m-1)!}\Pr\brac{d_n(j)\geq \D_\g^{{\vt}}\mid\eta_j}d\eta_j\\
&\lesssim \sum_{\k\geq1}\sum_{j\in J_\g^{{\vt},\k}}\frac{\eta_{{\vt}}}{2({\vt}j)^{1/2}}
\int_{\eta_j=(1-1/L)m}^{(1+1/L)m}\frac{\eta_j^{m-1}e^{-\eta_j}}{(m-1)!} \cdot \begin{cases}1&\k\leq 3,\\\exp\set{-\frac{d_5\g m\z({\vt})}{L^2}}&4\leq \k\leq 10L,\\ (e(L/\k)^{1/2})^{d_6\g m\z({\vt})} &\k>10L, \end{cases}\\
&\leq\sum_{\k\geq1} \frac{\eta_{{\vt}}}{\g L}\begin{cases}2&\k\leq 3,\\4e^{-d_5\g L}&4\leq \k\leq 10L,\\ \bfrac{eL^{1/2}}{\k^{1/2}}^{d_6\g L^3} &\k>10L, \end{cases}
\end{align}
\begin{align}
\sum_{\k\geq 1}\E(r_\g^{{\vt},\k,3}\mid{\Ht})&\leq \sum_{\k\geq 1}\sum_{j\in J_\g^{{\vt},\k}}\frac{\eta_{{\vt}}}{2({\vt}j)^{1/2}} \int_{\eta_j=(1+1/L)m}^{\infty}\frac{\eta_j^{m-1}e^{-\eta_j}}{(m-1)!}\\
&\leq \sum_{\k\geq 1}\sum_{j\in J_\g^{{\vt},\k}}\frac{\eta_{{\vt}}}{2({\vt}j)^{1/2}} e^{-m/(3L^2)},\qquad\text{ from Lemma \ref{lemeta}(d)},\\
&\leq \frac{\eta_{{\vt}}}{\g}e^{-m/(3L^2)}\bfrac{n}{{\vt}}^{1/2}\\
&\leq \frac{5^{1/2}\eta_{{\vt}}}{\g}e^{-m/(3L^2)}.
\end{align}
The above upper bounds are small enough to give the lemma in this case, without trouble.
\proofend

We are now in a position to prove \eqref{bk1}.  We confirmed the second part of the statement \eqref{bk1} above, using \eqref{RKxx}, so only the first part remains.  The first part follows immediately from Lemma \ref{34} and the following, by addition:
\begin{lemma}\label{newlem1}
$$\E(\r_t\mathbf{1}_{{\vtt}\geq {\vt}}\mid {\Ht})\leq \frac{21\eta_{{\vt}}}{mL}+\frac{L^3}{m^2}.$$
\end{lemma}
\begin{proof}We consider cases.\\
{\bf Case 1: $n_1\leq\vt\leq n\brac{1-\frac{3L^3}{m}}$:} 
 Then,
\beq{rho}
\E(\r_t\mathbf{1}_{{\vtt}\geq {\vt}}\mid {\Ht})\leq I_1+I_2+I_3+I_4,
\eeq
where
\begin{align}
I_1&=\int_{\g=1/8}^{5/4}\E(\r_t\mathbf{1}_{{\vtt}\geq {\vt}}\mid\g_{{\vt}}^*=\g)d\Pr(\g_\vt^*\leq \g),\\
&\leq \int_{\g=1/8}^{5/4}\brac{\frac{2\times 8^3}{mL}\times \eta_{{\vt}}\times \brac{\frac{7L}{8}+13+100Le^{-c_2\g L}}}d\Pr(\g_\vt^*\leq \g),\\
&\hspace{2in}\text{ by Remark \ref{remboo}, Lemma \ref{l.right}},\\
&\leq \frac{1000\eta_{{\vt}}}{m}\int_{\g=1/8}^{5/4}d\Pr(\g_\vt^*\leq \g),\\
&\leq\eta_{{\vt}}e^{-c_1m},\qquad\text{from \eqref{P<g0}}.\label{pp1}\\
I_2&=\int_{\g=5/4}^{10000}\brac{\frac{2}{m\g^3L}\times \eta_{{\vt}}\times (13+100Le^{-c_2\g L})}d\Pr(\g_\vt^*\leq \g),\\
&\leq \frac{20\eta_{{\vt}}}{mL}.\label{pp2}\\
I_3&=\int_{\g=10000}^\infty \brac{\frac{2}{m\g^3L}\times \eta_{{\vt}}\times (13+100Le^{-c_2\g L})}d\Pr(\g_\vt^*\leq \g)\\
&\leq \frac{27\eta_{{\vt}}}{10^{15}Lm}\int_{\g=100}^\infty \g^{-cm}d\g,\qquad \text{from \eqref{P>g}},\\
&=\frac{27\eta_{{\vt}}}{10^{15}Lm}\times\frac{1}{10^{2(cm-1)}(cm-1)}.\label{pp3}\\
I_4&= \int_{\g=0}^{1/8}\E(\r_t\mathbf{1}_{{\vtt}\geq {\vt}}\mid\g_{{\vt}}^*=\g)d\Pr(\g_\vt^*\leq \g)\\
&\leq e^{-d_0m^{1/2}},\label{pp4a}
\end{align}
To obtain the term $e^{-d_0m^{1/2}}$ in \eqref{pp4a} we use \eqref{P<g} and \eqref{P<gg} to obtain
\begin{align}
&\max\set{\g\in[0,1/8]: \g^{-2}\Pr(\g^*\leq\g)}\leq\\ 
&\max\set{\g\in[0,1/8]:(\g^{1/2-4/(c_1m^2)}e^{1-\g^{1/2}})^{c_1m^2}}+\\
&\hspace{2in}\max\set{\g\in[0,1/8]: m\g^{-2}\min\set{e^{-c_1m\g^{1/2}},e^{-c_1/\g^{1/2}}}}\\ 
&\leq (8^{4/(c_1m^2)-1/2}e^{1-1/8^{1/2}})^{c_1m^2}+m^2e^{-c_1m^{1/2}}.\label{Pgamma}
\end{align}
The first case of the lemma now follows from \eqref{pp1}, \eqref{pp2}, \eqref{pp3} and \eqref{pp4a}.

{\bf Case 2: $\vt>n\brac{1-\frac{3L^3}{m}}$:}\\ 
We observe first that $n\leq v_t\brac{1+\frac{4L^3}{m}}$. Then we let $Z=d_n(v_t)-m$ be the number of right neighbors of $v_t$. Furthermore,
\beq{97}
\E(Z\mid{\Ht})\lesssim\sum_{j={\vt}+1}^n\frac{w_i}{W_j}\lesssim \sum_{j={\vt}+1}^{v_t\brac{1+\frac{4L^3}{m}}}\frac{\eta_{{\vt}}}{2({\vt}j)^{1/2}}\lesssim \eta_{{\vt}}\brac{\brac{1+\frac{4L^3}{m}}^{1/2}-1}\leq \eta_{{\vt}}\frac{4L^3}{m}.
\eeq 
{\bf Case 2a: $\eta_{v_t}\geq 1/L^{1/2}$.}\\
We use \eqref{97} and Lemma \ref{lemd}(d) to prove
\beq{99}
\Pr\brac{Z\geq \frac{\eta_{{\vt}}}{L}\ \bigg|\ {\Ht}}\leq e^{-d_1\eta_{v_t}L}\leq e^{-d_1L^{1/2}}.
\eeq
Then we can write
\beq{100}
\E(\r_{t}\mathbf{1}_{{\vtt}\geq {\vt}}\mid{\Ht})\leq \brac{1+\frac{4L^3}{m}}\times e^{-d_1L^{1/2}} +\frac{2\eta_{{\vt}}}{Lm}\leq\frac{3\eta_{{\vt}}}{Lm}.
\eeq
{\bf Explanation:} $\r_{t}$ will be at most $\brac{1+\frac{4L^3}{m}}$ if the unlikely event in \eqref{99} occurs. Failing this, the chance that $\r_{t}>1$ is at most $\frac{2Z}{m}\leq \frac{2\eta_{{\vt}}}{Lm}$.

{\bf Case 2b: $\eta_{v_t}< 1/L^{1/2}$.}\\
It follows from \eqref{97} that $\E(Z\mid H_t)\lesssim 4L^{5/2}/m$. It then follows from Lemma \ref{lemd}(d) that 
$$\Pr\brac{Z\geq \frac{L^3}{3m}\bigg| H_t}\leq e^{-d_2L^{1/2}}.$$
We then have
$$E(\r_{t}\mathbf{1}_{{\vtt}\geq {\vt}}\mid{\Ht})\leq \brac{1+\frac{4L^3}{m}}\times e^{-d_2L^{1/2}} +\frac{2L^3}{3m^2}\leq \frac{L^3}{3m^2}.\qedhere $$
\end{proof}

\subsection*{\ref{step:4m}}   

We now prove \eqref{bk2}. To do this, we will obtain a recurrence for $\E(\eta_{{\vtt}}\mid{H_t})$, and, at the end, obtain the bound $4m$ by averaging over the possible histories $H_t$.

We begin by writing
\beq{2terms}
\E(\eta_{{{\vtt}}}\mid{H_t})=
\E(\eta_{{{\vtt}}}\mathbf{1}_{{{\vtt}}<{\vt}}\mid{H_t})+ \E(\eta_{{{\vtt}}}\vun{{{\vtt}}>{\vt}}\mid{H_t})
\eeq

We consider each term in \eqref{2terms} separately.  For the \textbf{first term}, since
\[
\eta_{{{\vtt}}}\mathbf{1}_{{{\vtt}}<{\vt}}\leq \max\set{\eta_l:\;1\leq l<{{\vt}},\,l\in N_L({{\vt}})}        \mathbf{1}_{{{\vtt}}<{\vt}}
\leq \max\set{\eta_l:\;1\leq l<{{\vt}},\,l\in N_L({{\vt}})},        
\]
we have that
\begin{align}
\E(\eta_{{\vtt}}\mathbf{1}_{{{\vtt}}<{\vt}}\mid{H_t})&\leq\E\brac{\max\set{\eta_l:\;1\leq l<{{\vt}},\,l\in N_L({{\vt}})}\mid{H_t}}\\
&=\int_{\eta=0}^\infty\Pr(\max\set{\eta_l:\;1\leq l<{{\vt}},\,l\in N_L({{\vt}})}\geq\eta\mid{H_t})d\eta\\
&=\int_{\eta=0}^\infty\Pr(\exists 1\leq l<{{\vt}},\,l\in N_L({{\vt}}):\;\eta_l\geq\eta\mid{H_t})d\eta\\
&\leq \int_{\eta=0}^\infty\sum_{l=1}^{{{\vt}}-1}\Pr(l\in N_L({{\vt}})\text{ and }\eta_l\geq\eta\mid{H_t})d\eta\\
&\lesssim \int_{\eta=0}^\infty\sum_{l=1}^{{{\vt}}-1}\frac{w_l}{W_{v_t}}\int_{\eta_l=\eta}^\infty \frac{\eta_l^{m-1}e^{-\eta_l}}{(m-1)!}d\eta_l d\eta\\
&\lesssim\sum_{l=1}^{{{\vt}}-1} \int_{\eta=0}^\infty\int_{\eta_l=\eta}^\infty \frac{\eta_l}{2m(l{{v_t}})^{1/2}}\frac{\eta_l^{m-1}e^{-\eta_l}}{(m-1)!}d\eta_ld\eta\\
&\lesssim{2m+(1+o(1))\int_{\eta=2m}^\infty\int_{\eta_l=\eta}^\infty \frac{\eta_l^{m}e^{-\eta_l}}{(m-1)!}d\eta_ld\eta}\\
&\lesssim{2m+\int_{\eta=2m}^\infty 4e^{-3\eta/10}d\eta},\text{ from Lemma \ref{lemeta}(c),}\\
&\lesssim{2m+20e^{-3m/5}}\\
&\leq 3m.\label{last0}
\end{align}

We now bound the \textbf{second term} of \eqref{2terms}.  We consider two cases, according to properties of the history $H_t$ (which determines $\vt$ and ${\eta_\vt}$).

\textbf{Case 1:} $H_t$ is such that ${{\vt}}\leq \brac{1-\frac{1}{\om^{1/2}}}n$.\\
In this case, we have that 
\begin{multline}
\E(\eta_{{{\vtt}}}\vun{{{\vtt}}>{\vt}}\mid{H_t})\\\leq
\E\brac{\max\set{\eta_l:\;{{\vt}}<l\leq n,\,{{\vtt}}\in N_L(l),d_n(l)\geq \D_{\g_{{\vt}}^*}^{{\vt}}}\vun{{{\vtt}}>{\vt}}\mid{\Ht}}\\\leq 
\E\brac{\max\set{\eta_l:\;{{\vt}}<l\leq n,\,{{\vtt}}\in N_L(l),d_n(l)\geq \D_{\g_{{\vt}}^*}^{{\vt}}}\mid{\Ht}}.
\end{multline}
So we have that
\begin{align}
&\E(\eta_{{\vtt}}\vun{{{\vtt}}>{\vt}}\mid\Ht)\\
&\leq\E\brac{\max\set{\eta_l:\;{{\vt}}<l\leq n,\,{{\vtt}}\in N_L(l),d_n(l)\geq \D_{\g_{{\vt}}^*}^{{\vt}}}\mid{\Ht}}\\
&=\int_{\eta=0}^\infty\Pr( \max\set{\eta_l:\;{{\vt}}<l\leq n,\,{{\vt}}\in N_L(l),d_n(l)\geq \D_{\g_{{\vt}}^*}^{{\vt}}}\geq\eta\mid{\Ht})d\eta\\
&=\int_{\eta=0}^\infty\Pr\brac{\exists {{\vt}}<l\leq n,\,{{\vt}}\in N_L(l):\;\eta_l\geq\eta,d_n(l)\geq \D_{\g_{{\vt}}^*}^{{\vt}}\mid{\Ht}}d\eta\\
&\leq \sum_{l={{\vt}}+1}^n\int_{\eta=0}^\infty\Pr\brac{({{\vt}}\in N_L(l))\wedge (\eta_l\geq \eta) \wedge (d_n(l)\geq \D_{\g_{{\vt}}^*}^{{\vt}}) \mid{\Ht},\eta_l}d\eta\\
&\lesssim\sum_{l={{\vt}}+1}^n\frac{\eta_{\vt}}{2(l{{\vt}})^{1/2}}\int_{\eta=0}^\infty\int_{\eta_l=\eta}^\infty \frac{\eta_l^{m-1}e^{-\eta_l}}{(m-1)!}\Pr\brac{d_n(l)\geq \D_{\g_{{\vt}}^*}^{{\vt}}\mid{\Ht},\eta_l}d\eta_ld\eta.\label{eq11}
\end{align}
Recall that in the final two lines, $v_t$ and $\eta_{v_t}$ are not random variables, but are the actual values of these random variables in the history $H_t$, so this is a deterministic upper bound on $\E(\eta_{{\vtt}}\vun{{{\vtt}}>{\vt}}\mid\Ht)$.

We split the sum in the RHS of \eqref{eq11} into $E_1+E_2+E_3+E_4$ according to the ranges of $l$ and $\eta$, and bound each separately. The first part consists of 
$$E_1=\sum_{l={\vt}+1}^n\frac{\eta_{{\vt}}\mathbf{1}_{l\leq 4m^2\vt/(\g_{{\vt}}^*)^2}}{2(l{\vt})^{1/2}} \int_{\eta=0}^{2m}\int_{\eta_l=\eta}^\infty \frac{\eta_l^{m-1}e^{-\eta_l}}{(m-1)!}\Pr\brac{d_n(l)\geq \D_{\g_{{\vt}}^*}^{{\vt}}\mid{\Ht},\eta_l}d\eta_ld\eta.$$

Even though $v_t$ and $\eta_{v_t}$ are constants (determined by $H_t$), we caution that $\g_{v_t}^*$ and so also $E_1$ are random variables.

Observe that we have that
$$\int_{\eta_l=\eta}^\infty \frac{\eta_l^{m-1}e^{-\eta_l}}{(m-1)!}\Pr\brac{d_n(l)\geq \D_{\g_{{\vt}}^*}^{{\vt}} \mid{\Ht},\eta_l}d\eta_l\leq\int_{\eta_l=\eta}^\infty \frac{\eta_l^{m-1}e^{-\eta_l}}{(m-1)!}d\eta_l\leq 1$$

which allows us to write
\beq{E1a}
E_1\mathbf{1}_{\g_{{\vt}}^*\leq 5/4}\leq \mathbf{1}_{\g_{{\vt}}^*\leq 5/4} \sum_{l={\vt}+1}^{n}\frac{2m\eta_{\vt}\mathbf{1}_{l\leq  4m^2\vt/(\g_{{\vt}}^*)^2}}{2(l{\vt})^{1/2}}\leq  \frac{5m^2\eta_{\vt}}{\g_{{\vt}}^*} \mathbf{1}_{\g_{{\vt}}^*\leq 5/4}.
\eeq
We will use this expression when we take the expectation over $\g_{{\vt}}^*\leq 5/4$.

We also have that
\beq{I1I2I3}
\int_{\eta_l=\eta}^\infty \frac{\eta_l^{m-1}e^{-\eta_l}}{(m-1)!}\Pr\brac{(d_n(l)\geq \D_{\g^*}^{{\vt}})\wedge (\g_{{\vt}}^*>5/4)\mid{\Ht},\eta_l}d\eta_l\leq I_1+I_2+I_3,
\eeq
where 
\begin{align}
I_1&=\int_{\eta_l=\eta}^{7m/8}\frac{\eta_l^{m-1}e^{-\eta_l}}{(m-1)!} d\eta_l\leq m\brac{\frac78e^{1/8}}^m\leq e^{-d_0m},\quad\text{from Lemma \ref{lemeta}(a)}, \label{p1}\\
I_2&=\int_{\eta_l=9m/8}^{\infty}\frac{\eta_l^{m-1}e^{-\eta_l}}{(m-1)!}d\eta_l\leq e^{-d_1m},\quad\text{from Lemma \ref{lemeta}(d)}.\label{p3}\\
I_3&=\int_{\eta_l=7m/8}^{9m/8}\frac{\eta_l^{m-1}e^{-\eta_l}}{(m-1)!}\Pr\brac{\brac{d_n(l)-m\geq\g_{{\vt}}^*m\z({\vt})}\wedge (\g_{{\vt}}^*>5/4)\bigg|{\Ht},\eta_l}d\eta_l, \\
&\leq \int_{\eta_l=7m/8}^{9m/8}\frac{\eta_l^{m-1}e^{-\eta_l}}{(m-1)!} \Pr\brac{d_n(l)-m\geq \frac{5}{4}m\z({\vt})\bigg| {\Ht},\eta_l}d\eta_l.\label{zetal}
\end{align}
We bound $I_3$ with two subcases:\\
{\bf Subcase 1a:} $\z(l)>0$.
\begin{align}
I_3&\leq\int_{\eta_l=7m/8}^{9m/8}\frac{\eta_l^{m-1}e^{-\eta_l}}{(m-1)!} \Pr\brac{d_n(l)-m\geq \frac{10\z({\vt})}{9\z(l)}\eta_l\z(l)\bigg| {\Ht},\eta_l}d\eta_l\:\:[\text{since }m\geq \frac 8 9 \eta_l],\\
&\leq \int_{\eta_l=7m/8}^{9m/8}\frac{\eta_l^{m-1}e^{-\eta_l}}{(m-1)!} \begin{cases}\exp\set{-\frac{(10\z({\vt})-9\z(l))^2\eta_l}{81\z(l)}}&10\z({\vt})\leq 18\z(l)\\\exp\set{-\frac{(10\z({\vt})-9\z(l))\eta_l}{27}}&10\z({\vt})>18\z(l)\end{cases}\quad[\text{from \eqref{HoHo} and } \ell\geq v_t+1],\\
&\leq \int_{\eta_l=7m/8}^{9m/8}\frac{\eta_l^{m-1}e^{-\eta_l}}{(m-1)!} \begin{cases}\exp\set{-\frac{7m\z({\vt})}{648}}&10\z({\vt})\leq 18\z(l)\\\exp\set{-\frac{7m\z({\vt})}{216}}&10\z({\vt})>18\z(l)\end{cases}\\
&\leq e^{-m\z({\vt})/100}.\label{p2}
\end{align}
{\bf Subcase 1b:} $\z(l)\leq0$.

In this case, we go back to \eqref{zetal} and use $\z^+(l)$ in place of $\z(l)$, see \eqref{zeta+x}.
\beq{est}
I_3\leq\int_{\eta_l=7m/8}^{9m/8}\frac{\eta_l^{m-1}e^{-\eta_l}}{(m-1)!} \Pr\brac{d_n(l)-m\geq \frac{10\z({\vt})}{9\z^+(l)}\eta_l\z^+(l)\bigg| {\Ht},\eta_l}d\eta_l,\\
\eeq
For $\e$ as in \eqref{eps} we see that $\z(l)\leq 0$ implies that $l\geq n(1-\e)^2$. In which case
\beq{asd1}
\z^+(l)\leq \frac{2\e}{1-\e}\leq 3\e.
\eeq
On the other hand, $v_t\leq 1-\frac{1}{\om^{1/2}}$ implies that
\beq{asd2}
\z(v_t)\geq \frac{1}{\om^{1/2}}-2\e\geq \frac{1}{2\om^{1/2}}.
\eeq
Comparing \eqref{asd1} and \eqref{asd2}, we see that $\z(v_t)\gg\z^+(l)$. From this and \eqref{HoHoHo} with $\b=\frac{10\z(v_t)}{9\z^+(l)}\geq \frac{1}{6\e\om^{1/2}}$ we deduce that 
\beq{asd3}
\Pr\brac{d_n(l)-m\geq \frac{10\z({\vt})}{9\z^+(l)}\eta_l\z^+(l)\bigg| {\Ht},\eta_l} \leq (6\e\om^{1/2})^{10\eta_l\z(v_t)}\leq (6\e\om^{1/2})^{35m\z(v_t)/36}.
\eeq
Plugging this estimate into \eqref{est} we obtain something stronger than \eqref{p2}, finishing Subcase 1b and giving that $I_3\leq e^{-m\zeta(\vt)/100}$ in all cases.

Having bounded the three terms in \eqref{I1I2I3}, we then have that
\begin{align}
E_1\mathbf{1}_{\g_{{\vt}}^*> 5/4}&\leq \sum_{l={\vt}+1}^{n}\frac{\eta_{v_t}\mathbf{1}_{l\leq 4m^2\vt/(\g_{{\vt}}^*)^2}}{(l{\vt})^{1/2}}\brac{e^{-d_2m}+ e^{-m\z({\vt})/100}}\\
&\leq \eta_{v_t}\brac{e^{-d_2m}\frac{5m}{\g_{{\vt}}^*}+\frac{e^{-m\z({\vt})/100}}{{\vt}^{1/2}} \sum_{l={\vt}+1}^{n}\frac{1}{l^{1/2}}}\\
&\leq \eta_{v_t}\brac{4me^{-d_2m}+e^{-m\z({\vt})/100}\cdot\frac{(n+1)^{1/2}-({\vt}+1)^{1/2}}{{\vt}^{1/2}}}\\
&\leq \eta_{v_t}(4me^{-d_2m}+2\z({\vt})e^{-m\z({\vt})/100})\\
&\leq \eta_{v_t}\brac{4me^{-d_2m}+\frac{200}{m}}.\label{E1b}
\end{align}
It follows from \eqref{E1a} and \eqref{E1b} that 
\beq{E1c}
E_1\leq \eta_{{\vt}}\brac{\frac{5m^2}{\g_{{\vt}}^*}\mathbf{1}_{\g_{{\vt}}^*\leq 5/4}+\brac{4me^{-d_2m}+\frac{200}{m}}}.
\eeq

We continue with the other parts of the RHS of \eqref{eq11}:
\begin{align}
E_2&=\sum_{l={\vt}+1}^n\frac{\eta_{{\vt}}\mathbf{1}_{l\leq 4m^2\vt/(\g_{{\vt}}^*)^2}}{2(l{\vt})^{1/2}}\int_{\eta=2m}^\infty \int_{\eta_l=\eta}^\infty \frac{\eta_l^{m-1}e^{-\eta_l}}{(m-1)!}\Pr(d_n(l)\geq \D_{\g^*}^{{\vt}}\mid{\Ht},\eta_l)d\eta_ld\eta\\
&\leq \sum_{l={\vt}+1}^{4m^2\vt/(\g_{{\vt}}^*)^2}\frac{\eta_{{\vt}}}{2(l{\vt})^{1/2}}\int_{\eta=2m}^\infty\int_{\eta_l=\eta}^\infty \frac{\eta_l^{m-1}e^{-\eta_l}}{(m-1)!}d\eta_ld\eta\label{E2a}\\
&\leq \sum_{l={\vt}+1}^{4m^2\vt/(\g_{{\vt}}^*)^2}\frac{\eta_{{\vt}}}{2(l{\vt})^{1/2}} \int_{\eta=2m}^\infty\bfrac{e\eta/m}{e^{\eta/m}}^md\eta\quad\text{from Lemma \ref{lemeta}(c)},\label{E2b}\\
&\leq \sum_{l={\vt}+1}^{4m^2\vt/(\g_{{\vt}}^*)^2}\frac{\eta_{{\vt}}}{2(l{\vt})^{1/2}}\times m\int_{x=2}^\infty e^{-3mx/10}dx\\
&= \sum_{l={\vt}+1}^{4m^2\vt/(\g_{{\vt}}^*)^2}\frac{10\eta_{{\vt}}e^{-3m/5}}{6(l{\vt})^{1/2}}\label{E2c}\\
&\leq \frac{e^{-d_3m}\eta_{{\vt}}}{\g_{{\vt}}^*}.\label{Om}
\end{align}
Note that we aborbed an $O(m)$ factor into the expression in \eqref{Om}. This is valid because $m$ is large. We continue to do this where possible.
\begin{align}
E_3&=\sum_{l={\vt}+1}^n\frac{\eta_{{\vt}}\mathbf{1}_{l> 4m^2\vt/(\g_{{\vt}}^*)^2}}{2(l{\vt})^{1/2}}\int_{\eta=0}^{\g_{{\vt}}^* (l/{\vt})^{1/2}}\int_{\eta_l=\eta}^\infty \frac{\eta_l^{m-1}e^{-\eta_l}}{(m-1)!}\Pr(d_n(l)\geq \D_{\g_{{\vt}}^*}^{{\vt}}\mid{\Ht})d\eta_ld\eta\\
&\leq\sum_{l=4m^2\vt/(\g_{{\vt}}^*)^2}^{n}\frac{\eta_{{\vt}}}{2(l{\vt})^{1/2}}\int_{\eta=0}^{\g_{{\vt}}^* (l/{\vt})^{1/2}}\int_{\eta_l=\g_{{\vt}}^* (l/{\vt})^{1/2}}^\infty \frac{\eta_l^{m-1}e^{-\eta_l}}{(m-1)!}d\eta_ld\eta\\
&\leq \sum_{l=4m^2\vt/(\g_{{\vt}}^*)^2}^n\frac{\eta_{{\vt}}}{2(l{\vt})^{1/2}}\int_{\eta=0}^{\g_{{\vt}}^* (l/{\vt})^{1/2}}\exp\set{-\frac{3\g_{{\vt}}^*(l/{\vt})^{1/2}}{10}}d\eta\quad\text{from Lemma \ref{lemeta}(c)},\\
&\leq \sum_{l=4m^2\vt/(\g_{{\vt}}^*)^2}^n\frac{\eta_{{\vt}}}{2(l{\vt})^{1/2}} \exp\set{-\frac{3\g_{{\vt}}^*(l/{\vt})^{1/2}}{10}},\\
&\leq \frac{\eta_{{\vt}}\g_{{\vt}}^*}{{\vt}}\int_{x=4m^2\vt/(\g_{{\vt}}^*)^2}^n\exp\set{-\frac{3\g_{{\vt}}^*(x/{\vt})^{1/2}}{10}}dx,\\
&\leq\frac{\eta_{{\vt}}\g_{{\vt}}^*}{{\vt}}\times\frac{8\vt}{(\g_{{\vt}}^*)^2}\int_{y=m}^\infty ye^{-3y/5}dy,\\
&= \frac{\eta_{{\vt}}e^{-d_4m}}{\g_{{\vt}}^*}.\\
E_4&=\sum_{l=i}^n\frac{\eta_{{\vt}}\mathbf{1}_{l>4m^2\vt/(\g_{{\vt}}^*)^2}}{2(l{\vt})^{1/2}}\int_{\eta=\g_{{\vt}}^* (l/{\vt})^{1/2}}^\infty\int_{\eta_l=\eta}^\infty \frac{\eta_l^{m-1}e^{-\eta_l}}{(m-1)!}\Pr(d_n(l)\geq \D_{\g_{{\vt}}^*}^{{\vt}}\mid{\Ht},\eta_l)d\eta_ld\eta\\
&\leq\sum_{l=4m^2\vt/(\g_{{\vt}}^*)^2}^n\frac{\eta_{{\vt}}}{2(l{\vt})^{1/2}}\int_{\eta=\g_{{\vt}}^* (l/{\vt})^{1/2}}^\infty\int_{\eta_l=\eta}^\infty \frac{\eta_l^{m-1}e^{-\eta_l}}{(m-1)!}d\eta_ld\eta\\
&\leq \sum_{l=4m^2\vt/(\g_{{\vt}}^*)^2}^n\frac{\eta_{{\vt}}}{2(l{\vt})^{1/2}}\int_{\eta=\g_{{\vt}}^* (l/{\vt})^{1/2}}^\infty e^{-3\eta/10}d\eta\\
&\leq \sum_{l=4m^2\vt/(\g_{{\vt}}^*)^2}^n\frac{5\eta_{{\vt}}}{3(l{\vt})^{1/2}} \exp\set{-\frac{3\g_{{\vt}}^*(l/{\vt})^{1/2}}{10}},\\
&\leq \frac{2\eta_{{\vt}}}{i^{1/2}}\int_{x=4m^2\vt/(\g_{{\vt}}^*)^2}^\infty x^{-1/2}\exp\set{-\frac{3\g_{{\vt}}^*(x/j)^{1/2}}{10}}dx,\\
&=\frac{4\eta_{{\vt}}}{\g_{{\vt}}^*}\int_{y=2m}^\infty e^{-3y/10}dy,\\
&\leq \frac{\eta_{{\vt}}e^{-d_5m}}{\g_{{\vt}}^*}.\
\end{align}
Thus, 
\beq{E1E2}
\E(\eta_{{\vtt}}\mathbf{1}_{{\vtt}>{\vt}}\mid{\Ht})\leq E_1+E_2+E_3+E_4\leq \begin{cases}\brac{\frac{5m^2}{\g_{{\vt}}^*}+\frac{e^{-d_6m}}{\g_{{\vt}}^*}}\eta_{{\vt}}\leq \frac{7m^2}{\g_{{\vt}}^*}\eta_\vt&\g_{{\vt}}^*\leq 5/4\\ \brac{\frac{e^{-d_7m}}{\g_{{\vt}}^*}+\frac{200}{m}}\eta_{{\vt}} &\g_{{\vt}}^*>5/4.\end{cases}
\eeq
We now integrate with respect to the value of $\g_\vt^*$.   (Note that $\g_\vt^*$ is actually a discrete random variable, so that $\Pr(\g_\vt^*\leq \g\mid H_t)$ is discontinuous, but one can view this as a Riemann-Stieltjes integral.  We write $\Pr^\dagger(\g_\vt^*\leq \g)$ below in place of $\Pr(\g_\vt^*\leq \g\mid H_t)$.)  Using Lemma \ref{Pg<} we see that if $m$ is large then integrating over $\g$,
\begin{align}
&\E(\eta_{{\vtt}}\mathbf{1}_{{\vtt}>{\vt}}\mid{\Ht})\\
&\leq \eta_{{\vt}}\brac{\int_{\g=0}^{5/4} \frac{7m^2}{\g}d\Pr^\dagger(\g_\vt^*\leq \g)+ \int_{\g=5/4}^\infty\frac{e^{-d_7m}}{\g}d\Pr^\dagger(\g_\vt^*\leq \g)+\frac{200}{m}}\\
&\leq \eta_{{\vt}}\brac{\int_{\g=0}^{5/4} \frac{7m^2}{\g}d\Pr^\dagger(\g_\vt^*\leq \g)+ e^{-d_8m}+\frac{200}{m}}\\
&=\eta_{{\vt}}\brac{\int_{\g=0}^{1/m} \frac{7m^2}{\g}d\Pr^\dagger(\g_\vt^*\leq \g)+\int_{\g=1/m^{1/2}}^{5/4} \frac{7m^2}{\g}d\Pr^\dagger(\g_\vt^*\leq \g) +e^{-d_8m}+\frac{200}{m}}\\
&\leq\eta_{{\vt}}\left(\left[\frac{7m^2}{\g}\Pr^\dagger(\g_{{\vt}}^*\leq \g)\right]_0^{1/m} +\int_{\g=0}^{1/m^{1/2}}\frac{7m^2}{\g^2}\Pr^\dagger(\g_{{\vt}}^*\leq \g)d\g\right.\\
&\hspace{2in}\left.+\int_{\g=1/m}^{5/4} \frac{7m^2}{\g}d\Pr^\dagger(\g_\vt^*\leq \g) +e^{-d_8m}+\frac{200}{m}\right)\\
&\leq \eta_{{\vt}}\left(e^{-d_9m^{1/2}}+\int_{\g=0}^{1/m}\frac{7m^2}{\g^2}\Pr^\dagger(\g_{{\vt}}^*\leq \g)d\g+\right.\\
&\hspace{1in}\left.\int_{\g=1/m}^{5/4} \frac{7m^3}{\g}d\Pr^\dagger(\g_\vt^*\leq \g) +e^{-d_8m}+\frac{200}{m}\right),\qquad\text{ from \eqref{P<gg}},\\
&\leq  \eta_{{\vt}}\brac{e^{-d_9m^{1/2}}+e^{-d_{10}m^{1/2}}+
7m^{4}\Pr^\dagger\brac{\frac 1 {m}\leq \g_\vt^*\leq 5/4}
+e^{-d_8m}+\frac{200}{m}},\\
&\leq  \eta_{{\vt}}\brac{e^{-d_9m^{1/4}}+e^{-d_{10}m^{1/2}}+
7m^{4}e^{-c_1m}+e^{-d_8m}+\frac{200}{m}},\qquad\text{ from \eqref{P<g0}}\\
&\leq \eta_{{\vt}}\brac{e^{-d_{12}m^{1/4}}+\frac{200}{m}}.\label{last1}
\end{align}
Combining \eqref{last0} and \eqref{last1} via \eqref{2terms}, we have that
\begin{align}
\E(\eta_{{\vtt}}\mid{H_t})&\leq 3m+\brac{e^{-cm^{1/4}}+\frac{200}{m}}\eta_{{\vt}}.\label{htbound}
\end{align}
This completes Case 1.  Case 2 is much shorter.

{\bf Case 2:} $H_t$ is such that ${{\vt}}> \brac{1-\frac{1}{\om^{1/2}}}n$.
\begin{align}
\E(\eta_{{\vtt}}\mid{\Ht})&\lesssim \sum_{l={\vt}+1}^n\frac{\eta_{{\vt}}}{2(l{\vt})^{1/2}}\int_{\eta=0}^\infty\int_{\eta_l=\eta}^\infty \frac{\eta_l^{m-1}e^{-\eta_l}}{(m-1)!}d\eta_ld\eta\\
&\sim\sum_{l={\vt}+1}^n\frac{\eta_{{\vt}}}{2(l{\vt})^{1/2}} \int_{\eta=0}^\infty e^{-\eta}\sum_{i=1}^{m}
\frac{\eta^{m-i}}{(m-i)!}d\eta\\
&\sim\sum_{l={\vt}+1}^n\frac{m\eta_{{\vt}}}{2(l{\vt})^{1/2}}\\
&\lesssim m\eta_{{\vt}}\frac{(n+1)^{1/2}-({\vt}+1)^{1/2}}{({\vt}+1)^{1/2}}\\
&\leq \frac{m\eta_{{\vt}}}{\om^{1/2}}.
\end{align}

This completes Case 2.  In particular, for sufficiently large $n$ we see that for any typical $H_t$ (i.e., in both Case 1 and Case 2), the bound from \eqref{htbound} is valid. Putting 
$$\cE_t=\cE\cap \set{H_t\text{ is typical}}$$ 
we deduce from \eqref{htbound} that 
\begin{align}
\E(\eta_{{\vtt}}\mid\cE_t)&\leq
3m+\brac{e^{-cm^{1/4}}+\frac{200}{m}}\E(\eta_{{\vt}}\mid\cE_t)\label{last2a}\\ &\lesssim
3m+\brac{e^{-cm^{1/4}}+\frac{200}{m}}\E(\eta_{{\vt}}\mid\cE_{t-1}).\label{last2}
\end{align}
We obtain \eqref{last2} from \eqref{last2a} because $\cE_t\subseteq \cE_{t-1}$ and so  
$$\E(\eta_{{\vt}}\mid\cE_{t-1})\geq\E(\eta_{{\vt}}\mid\cE_t)\Pr(\cE_t\mid\cE_{t-1})\sim \E(\eta_{{\vt}}\mid \cE_t).$$

Because $m$ is large, \eqref{bk2} will follow by induction once we have shown that 
\beq{ev1}
\E(\eta_{v_1}\ff)\leq 3m.
\eeq 
Here we will use the assumption that $v_1$ is chosen exacty according to the stationary distribution for a simple random walk on $G_n$.
In particular, we have
$$\Pr(\eta_{v_1}\geq \eta\ff)\leq \E\brac{\sum_{i=1}^n\frac{d_n(i)}{2mn}1_{\eta_{v_1}\geq \eta}},$$
and Lemma \ref{lemd} implies that if $\eta\geq 2m$
$$\E((d_n(i)-m)1_{\eta_{v_1}\geq \eta})\lesssim \bfrac{n}{i}^{1/2}\int_{\eta=2m}^\infty\frac{\eta^me^{-\eta}}{(m-1)!}d\eta\lesssim \bfrac{n}{i}^{1/2}\times 4me^{-\eta/2}.$$
Furthermore,
$$\E(m1_{\eta_{v_t}\geq \eta}\ff)=m\Pr(\eta_{v_1}\geq \eta)\leq 5me^{-\eta/2}.$$
So, if $\eta\geq 2m$, then
$$\Pr(\eta_{v_1}\geq \eta\ff)\lesssim \frac{9e^{-\eta/2}}{2n^{1/2}}\sum_{i=1}^n\frac{1}{i^{1/2}}\leq 10e^{-\eta/2}.$$
Therefore,
$$\E(\eta_{v_1}\ff)\leq 2m+\int_{\eta=2m}^\infty\Pr(\eta_{v_1}\geq \eta\ff)d\eta\leq 2m+10\int_{\eta=2m}^\infty e^{-\eta/2}d\eta$$
and \eqref{ev1} follows.\qed
\subsection{Exiting the main loop with success}\label{exiting}
In summary, it follows that w.h.p. {\sc DCA} reaches Step 7 in $O(\om\log n)$ time. Also, at this time $v_T\leq \log^{1/49}n$. This follows from Lemma \ref{lemd}(g), \ref{lemd}(h) and \ref{P.wsc}. Furthermore, this justifies using $n_1$ as a lower bound on vertices visited during the main loop. The random walk of Step 8 will w.h.p. take place on $[\log^{1/9}n]$. This follows from Lemma \ref{lemd}(j). Vertex 1 will be in the same component as ${\vt}$ in the subgraph of $G_n$ induced by vertices of degree at least $\frac{n^{1/2}}{\log^{1/20}n}$. This is because there is a path from $v_T$ to vertex 1 through vertices in $[v_T]$ only and furthermore it follows from Lemma \ref{lemd}(i) that w.h.p. every vertex on this path has degree at least $\frac{n^{1/2}}{\log^{1/20}n}$. The expected time to visit all vertices of a graph with $\n$ vertices is $O(\n^3)$, see for example Aleliunas, Karp, Lipton, Lov\'asz and Rackoff \cite{AKLLR}. Consequently, vertex 1 will be reached in a further $O((\log^{1/9}n)^3)=o(\log n)$ steps w.h.p, completing the proof of Theorem \ref{th0}.\qed

\section{Concluding remarks}
We have described an algorithm that finds a distinguished vertex quickly and which is local in a strong sense. There are some natural questions that are left unanswered:
\begin{itemize}
\item Can the running time be improved from $O(\om\log n)$ to $O(\log n)$? 
\item Can we get polylog expected running time for DCA if $m=2$?
\item Can we extend the analysis to other more general models of web graphs e.g. Cooper and Frieze \cite{CF1}. In this case, we would not be able to use the model described in Section \ref{condmodel}.
\end{itemize}

As a final observation, the algorithm {\sc DCA} could be used to find the vertex of largest degree: if we replace Step 8 by ``Do the random walk for $\log n$ steps and output the vertex of largest degree
encountered'' then w.h.p. this will produce a vertex of highest degree. This is because $\log n$ will be enough time to visit all vertices $v\leq \log^{1/39}n$, where the maximum degree vertex lies.

\appendix
\section{Proofs of properties \ref{P.upsx}--\ref{P.Px}}
In this section we give proofs of \ref{P.upsx}--\ref{P.Px}, which we list here for convenience.
\begin{enumerate}[label=\textbf{(P\arabic*)}]
\item \label{P.upsx} For $\Upsilon_{k,\ell}=\Upsilon_k-\Upsilon_\ell$, we have 
\beq{ups1x}
\Upsilon_{k,\ell}\in (k-\ell)\left[1\pm
    \frac{L\th_{k,\ell}^{1/2}}{3(k-\ell)^{1/2}}\right]
\eeq
for $(k,\ell)=(mn+1,0)$ or 
$$\frac{k-\ell}{m}\in \set{\om,\om+1,\ldots,n}\text{ and }k-l\geq \begin{cases}1&l=0\\ \log^2n&k\geq \log^{30}n,l>0\\ \log^{1/300}n&0<l<k<\log^{30}n.\end{cases}$$
Here $n_0=\frac{\l_0^2n}{\om\log^2n}$, $\l_0= \frac{1}{\log^{20/m} n}$,
$$\th_{k,\ell}=\begin{cases}\log k&\om\leq l< k\leq \log^{30}n,l>1\\k^{1/2}&\om\leq k\leq n^{2/5},l=0\\(k-\ell)^{1/2}&\log^{30}n<k\leq n^{2/5}\\ \frac{(k-\ell)^{3/2}\log n}{n^{1/2}}&n^{2/5}<k\leq n_0.\\ \frac{n}{\om^{3/2}\log^{2}n}&n_0<k.\end{cases}$$
\item \label{P.wsax}
$\displaystyle W_i\in \bfrac{{i}}{n}^{1/2}\left[1\pm \frac{L\th_i^{1/2}}{i^{1/2}}\right]\sim \bfrac{i}{n}^{1/2}\text{ for }\om\leq i\leq n.$
\item \label{P.wsbx}
$\displaystyle w_i\sim\frac{\eta_i}{2m(in)^{1/2}}\text{ for }\om\leq i\leq n.$
\item $\displaystyle \l_0\leq \eta_i\leq 40m\log\log n\text{ for }i\in[\log^{30}n]$.\label{P.wscx}
\item $\eta_i\leq \log n$ for $i\in [n]$.\label{P.Px}
\end{enumerate}
\subsection*{Proof of \ref{P.upsx}}\label{lem2}
Applying Lemma \ref{lemeta}(d),(e) to \eqref{eq1} for $i\ge1$ we
see that
\begin{align*}
&\Pr(\neg \ref{P.upsx})\\
&\leq 2\sum_{k=\om}^n\exp\set{-\frac{L^2\th_{k,0}}{27}}+ 2\sum_{k-\ell=\log^{1/300}n}^n\exp\set{-\frac{L^2\th_{k,\ell}}{27}} +2\exp\set{-\frac{L^2\th_{mn+1,0}}{27}}\\
&=2\sum_{k=\om}^{n^{2/5}}\exp\set{-\frac{L^2k^{1/2}}{27}}+
2\sum_{k=n^{2/5}+1}^{n_0}\exp\set{-\frac{L^2k^{3/2}\log n}{27n^{1/2}}}
+2\sum_{k=n_0+1}^{n+1}\exp\set{-\frac{L^2n}{27\om^{3/2}\log^{2}n}}\\
&+2\sum_{k-\ell=\log^{1/300}n}^{\log^{30}n}\exp\set{-\frac{L^2\log k}{27}} +2\sum_{k-\ell=\log^{30}n}^{n^{2/5}}\exp\set{-\frac{L^2(k-\ell)^{1/2}}{27}}+\\
&\ 2\sum_{k-\ell=n^{2/5}+1}^{n_0}\exp\set{-\frac{L^2(k-\ell)^{3/2}\log n}{27n^{1/2}}}
+2\sum_{k-\ell=n_0+1}^{n}\exp\set{-\frac{L^2n}{27\om^{3/2}\log^{2}n}}\\
&=o(1).
\end{align*}

\subsection*{Proof of \ref{P.wsax}}
For this we use
$$W_i=\bfrac{\Upsilon_{mi}}{\Upsilon_{mn+1}}^{1/2}.$$
Then,
\beq{Wi1}
W_i\notin \bfrac{{i}}{n}^{1/2}\left[1\pm \frac{L\th_i^{1/2}}{i^{1/2}}\right]
\eeq
implies that either 
$$\Upsilon_{mn+1}\notin (mn+1)\left[1\pm
  \frac{L\th_i^{1/2}}{3(mn+1)^{1/2}}\right]\text{ or }\Upsilon_{mi}\notin mi\left[1\pm \frac{L\th_i^{1/2}}{3i^{1/2}}\right].$$
These events are ruled out w.h.p. by \ref{P.upsx}.
\subsection*{Proof of \ref{P.wsbx}}
We use $(1+x)^{1/2}\leq 1+\frac{x}{2}$ for $0\leq |x|\leq 1$. Then,
\begin{align}
w_i&=\bfrac{\Upsilon_{mi}}{\Upsilon_{mn+1}}^{1/2}- \bfrac{\Upsilon_{m(i-1)}}{\Upsilon_{mn+1}}^{1/2}\\
&=\bfrac{\Upsilon_{m(i-1)}}{\Upsilon_{mn+1}}^{1/2} \brac{\brac{1+\frac{\eta_{{i}}}{\Upsilon_{m(i-1)}}}^{1/2}-1}\\
&\leq
\frac{\brac{m(i-1)\brac{1+\frac{L\th_i^{1/2}}{3m^{1/2}(i-1)^{1/2}}}}^{1/2}}
{\brac{(mn+1)\brac{1-\frac{L\th_i^{1/2}}{3(mn+1)^{1/2}}}}^{1/2}}
\frac{\eta_{{i}}}{2m(i-1)\brac{1-\frac{L\th_i^{1/2}}{3m^{1/2}(i-1)^{1/2}}}}\\
&\leq \frac{\eta_{{i}}}{2m(in)^{1/2}}\brac{1+\frac{2L\th_i^{1/2}}{m^{1/2}i^{1/2}}}.\label{wi1}
\end{align}
A similar calculation gives
\beq{wi2}
w_i\geq\frac{\eta_{{i}}}{2m(in)^{1/2}} \brac{1-\frac{2L\th_i^{1/2}}{m^{1/2}i^{1/2}}}.
\eeq
\subsection*{Proof of \ref{P.wscx}}
The upper bound follows from Lemma \ref{lemeta}(c). For the lower bound, we observe by \eqref{trivial} that the expected number of $i\leq \log^{30}n$ with $\eta_i\leq \l_0$ is at most $\log^{30}n\times \l_0^m=o(1)$.
\subsection*{Proof of \ref{P.Px}}
This follows from Lemma \ref{lemeta}(c).
\section{Proof of Lemma \ref{lemd}}
We restate the lemma for convenience.
\begin{lemma}\label{lemdx}\ 
\begin{enumerate}[label=(\alph*)]
\item If $\cE$ occurs then $\bar{d}_n-m\in [\eta_i\z(i),\eta_i\z^+(i)].$
\item $\Pr(d_n(i)-m\leq (1-\a)m\z(j))\leq e^{-\a^2\eta_i\z(i)/2}$ for $0\leq\a\leq 1$.
\item $\Pr(d_n(i)-m\geq (1+\a)m\z^+(j))\leq e^{-\a^2\eta_i\z^+(i)/3}$ for $0\leq\a\leq 1$.
\item $\Pr(d_n(i)-m\geq \b m\z^+(j))\leq (e/\b)^{\b\eta_i\z^+(i)}$ for $\b\geq 2$.
\item W.h.p. $\eta_i\geq \l_0$ and $\om\leq i\leq n^{1/2}$ implies that  $d_n(i)\sim\eta_i\bfrac{n}{{i}}^{1/2}$.
\item W.h.p. $\om\leq i\leq \log^{30}n$ implies that  $d_n(i)\sim\eta_i\bfrac{n}{{i}}^{1/2}$.
\item  W.h.p. $\om\leq i\leq n^{1/2}$ implies that $d_n(i)\lesssim \max\set{1,\eta_i}\bfrac{n}{{i}}^{1/2}.$
\item  W.h.p. $n^{1/2}\leq i\leq n \text{ implies } d_n(i)\leq n^{1/3}.$
\item W.h.p. $1\leq i\leq \log^{1/49}n$ implies that $d_n(i)\geq \frac{n^{1/2}}{\log^{1/20}n}$.
\item W.h.p. $d_n(i)\geq \frac{n}{\log^{1/20}n} \text{ implies } i\leq \log^{1/9}n$.
\end{enumerate}
\end{lemma}
\proofstart
(a) Suppose that we fix the values for $W_1,W_2,\ldots,W_n$. Then the degree $d_n(i)$ of vertex $i$ can be expressed
$$d_n(i)=m+\sum_{j=i}^n\sum_{k=1}^m\z_{j,k}$$
where the $\z_{j,k}$ are independent Bernouilli random variables such that 
$$\Pr(\z_{j,k}=1)\in \left[\frac{w_i}{W_j},\frac{w_i}{W_{j-1}}\right].$$
So, putting
$$\bar{d}_n(i)=\E(d_n(i)\ff)$$
we have
$$mw_i\sum_{j=i}^{n}\frac{1}{W_j}\leq\bar{d}_n(i)-m \leq mw_i\sum_{j=i-1}^{n}\frac{1}{W_j}.$$
Now assuming that \ref{P.wsa}, we have for $\om\leq i\leq n$, 
\beq{f1}
\sum_{j=i}^n\frac{1}{W_j}\geq \sum_{j=i}^n\bfrac{n}{j}^{1/2}\brac{1-\frac{2L\th_j^{1/2}}{j^{1/2}}}.
\eeq
But
\begin{align}
\sum_{j=\om}^n\frac{\th_j^{1/2}}{j}&\leq \sum_{j=\om}^{n^{2/5}}\frac{1}{j^{3/4}}+\sum_{j=n^{2/5}+1}^{n_0}\frac{\log^{1/2}n}{n^{1/4}j^{1/4}} +\sum_{j=n_0+1}^n\frac{n^{1/2}}{j\om^{3/4}\log n} \nonumber\\
&\leq 4n^{1/10}+\frac{4n^{1/2}}{3\om^{3/4}\log n}+\frac{3n^{1/2}\log\log n}{\om^{3/4}\log n} \label{large}\\
&\leq \frac{4n^{1/2}\log\log n}{\om^{3/4}\log n}.\nonumber
\end{align}
It follows that 
\begin{align*}
\bar{d}_n(i)&\geq m+mw_in^{1/2}\brac{2(n^{1/2}-({i}+1)^{1/2})-\frac{9Ln^{1/2}\log\log n}{\om^{3/4}\log n}}\\
&\geq m+\eta_i\bfrac{n}{{i}}^{1/2} \brac{1-\bfrac{{i}}{n}^{1/2}-\frac{1}{n^{1/2}i^{1/2}}-\frac{9L\log\log n}{2\om^{3/4}\log n} },
\end{align*}
after using \ref{P.wsb}.

A similar calculation gives a similar upper bound for $\bar{d}_n(i)$ and this proves that 
\beq{bardv0}
i\geq \om\text{ implies that }\bar{d}_n(i)\in m+\eta_i\bfrac{n}{{i}}^{1/2}\left[1-\bfrac{{i}}{n}^{1/2}\pm \frac{5L\log\log n}{\om^{3/4}\log n}\right].
\eeq
It follows from \eqref{HoHo} and \eqref{HoHoa} that 
\begin{align}
\Pr\brac{d_n(i)-m\leq (1-\a)\eta_i\z(i)\bigg|\ \eta_i}&\leq 
\exp\set{-\frac{L^2\eta_in^{1/2}}{4i^{1/2}\om^{1/2}}}.\\
\Pr\brac{d_n(i)-m\geq (1+\a)\eta_i\z(i)\bigg|\ \eta_i}&\leq 
\exp\set{-\frac{L^2\eta_in^{1/2}}{4i^{1/2}\om^{1/2}}}.\\
\end{align}
(a) For $\eta_i\geq \l_0$ and $\om\leq i\leq n_0$ we have 
$$\exp\set{-\frac{L^2\eta_in^{1/2}}{4i^{1/2}\om^{1/2}}}\leq e^{-L^2\log n/4}.$$
 
(b) This follows from (a) and \eqref{HoHoa}.

(c) This follows from (a) and \eqref{HoHo}.

(d) This follows from (a) and \eqref{HoHoHo}.

(e) This follows from (a), (b), (c) and \eqref{smalli}.

(f) This follows from (e) and \ref{P.wsc}.

(g) This follows from (c) and \eqref{smalli}.

(h) The degree of $i\geq n^{1/2}$ is stochastically dominated by the degree of $n^{1/2}$. Also, the probability that $d_n(n^{1/2})$ exceeds the stated upper bound is $o(1/n)$. So (h) follows from (g).

(i) For $\om\leq i\leq \log^{1/49}n$, this follows from (f) and  \ref{P.wsc}. For $1\leq i<\om$ we can use (b) with $\eta_i\geq \l_0$ and $\a=n^{-1/10}$.

(j) This follows from (e), (f) and (g) and \ref{P.wsc}.
\proofend
\end{document}